\documentclass[a4paper,11pt]{amsart}

\usepackage[utf8]{inputenc}
\usepackage[T1]{fontenc}
\usepackage{amsfonts}
\usepackage{amsthm}
\usepackage{amsmath}
\usepackage{tensor}
\usepackage[english]{babel}
\usepackage[all]{xy}
\usepackage{graphicx}
\usepackage{amscd}
\usepackage{latexsym}
\usepackage{hyperref}
\usepackage{mathrsfs}
\usepackage{enumerate}

\usepackage{mathabx,epsfig}

\usepackage{color}
\usepackage{xcolor}

\usepackage{calc}

\theoremstyle{definition}
 \newtheorem{defi}{Definition}[section]

\theoremstyle{remark}
 \newtheorem{remark}[defi]{Remark}
 \newtheorem{exam}[defi]{Example}

\theoremstyle{plain}
\newtheorem{theo}[defi]{Theorem}
\newtheorem{prop}[defi]{Proposition}

\newtheorem{assumption}[defi]{Assumption}
\newtheorem*{acknow}{Acknowledgments}

\def\e#1\e{\begin{equation}#1\end{equation}}
\def\ea#1\ea{\begin{align}#1\end{align}}

\numberwithin{equation}{section}

\newcommand{\zz}{\mathbb{Z}}

\newcommand{\rr}{\mathbb{R}}
\newcommand{\cc}{\mathbb{C}}

\newcommand{\Z}[1]{\zz_{#1}}
\newcommand{\abs}[1]{\left\vert #1 \right\vert}

\newcommand{\Bcal}{\mathcal{B}}

\newcommand{\Ecal}{\mathcal{E}}
\newcommand{\Fcal}{\mathcal{F}}

\newcommand{\Hcal}{\mathcal{H}}
\newcommand{\Ical}{\mathcal{I}}
\newcommand{\Jcal}{\mathcal{J}}

\newcommand{\Mcal}{\mathcal{M}}

\newcommand{\Pcal}{\mathcal{P}}

\newcommand{\Scal}{\mathcal{S}}
\newcommand{\Tcal}{\mathcal{T}}
\newcommand{\Ucal}{\mathcal{U}}

\newcommand{\Ctil}{\widetilde{C}}

\newcommand{\acs}{almost complex structure}
\newcommand{\Ainf}{$A_\infty$}

\title[Equivariant Floer homology]{Equivariant Lagrangian Floer homology via multiplicative flow trees}
\author[Guillem Cazassus]{Guillem Cazassus}

\address{Centre for Quantum Mathematics, Department of Mathematics and Computer Science, University of Southern Denmark}

\email{g.cazassus@gmail.com}

\thanks{This work was funded by EPSRC grant reference EP/T012749/1, the Simons Collaboration grant no. 994320, and the ERC-SyG project ReNewQuantum}

\begin{document}

\begin{abstract}We provide constructions of equivariant Lagrangian Floer homology groups, by constructing and exploiting an \Ainf -module structure on the Floer complex.
\end{abstract}

\maketitle
\tableofcontents

\section{Introduction}
\label{sec:intro}

\subsection{Background}
\label{ssec:background}

Equivariant homology of a $G$-space $X$ is a rich invariant that is usually better behaved than the homology of the quotient, when the action is not free. 

Likewise, if a symplectic manifold $M$ admits a Lie group action, one can define equivariant versions of Floer homology \cite{Frauenfelder,Woodward_quasimap_FH,SeidelSmith,HLSflex,HLSflex_corr,HondaBao,ChoHong,HLSsimplicial,KimLauZheng,LauLeungLi}
 (see for example \cite{equiv} for a more detailed account of the existing approaches). 

In particular, for a pair of Lagrangians $L_0, L_1$ in a symplectic manifold with an action of a Lie group, the first construction of equivariant Lagrangian Floer homology appeared in \cite{HLSsimplicial}. Compared to other approaches, this one is of algebraic nature, and relies on advanced constructions from the theory of $(\infty,1)$-categories. 

In this paper, we provide another construction. It is similar in nature with \cite{HLSsimplicial}, in that our contruction relies on an algebraic action of $G$ on the (non-equivariant) Floer complex $CF(L_0, L_1)$. Though, the algebraic tools involved are comparatively simpler, and perhaps more standard to the symplectic topology community. While in \cite{HLSsimplicial} the algebraic object associated with $G$ is a simplicial nerve, we instead use the Morse complex of $G$, endowed with a certain \Ainf -algebra structure built from the group multiplication.

On the other side of the Atiyah-Floer conjecture, namely instanton gauge theory, Miller Eismeier \cite{mike_equiv} also constructed several versions of equivariant Floer homology groups. Exploiting a dga action on the Floer complex, he produced  Borel, co-Borel and Tate homology groups, by applying some Bar constructions he developped.

\subsection{Statement of results}
\label{ssec:stat_results}

\begin{theo}\label{th:Ainf_str_intro}Under certain assumptions (see Section~\ref{ssec:geom_setting}), if a compact Lie group $G$ acts by symplectomorphisms on a symplectic manifold $M$, and $L_0, L_1\subset M$ is a pair of $G$-invariant Lagrangians, then:
\begin{itemize}
\item the Morse complex $CM_*(G)$ can be endowed with an \Ainf -algebra structure,
\item the Floer complex $CF(L_0, L_1)$ is an \Ainf -module over the former \Ainf -algebra.
\end{itemize} 
\end{theo}

Transposing \cite[App.~A]{mike_equiv} to the \Ainf\ setting, we then obtain:

\begin{theo}\label{th:4groups_intro}
For a certain dga $A_G$ associated to $G$, we construct four dg modules respectively refered to as the Borel, co-Borel, twisted Borel, and Tate complexes:
\e
CF^{+}_G(L_0, L_1),\ CF^{-}_G(L_0, L_1),\ \widetilde{CF}^{+}_G(L_0, L_1),\ CF^{\infty}_G(L_0, L_1).
\e
Furthermore, the three later fit into a short exact sequence of dg modules, inducing a long exact sequence of $H_*(A_G)$-modules in homology.
\end{theo}

\begin{remark}Following \cite{mike_equiv}, the dga $A_G$ should be equivalent to $C^*(BG)$, and the twisted version $\widetilde{CF}^{+}_G(L_0, L_1)$ to $CF^{+}_G(L_0, L_1)$, we will address these two points later.
\end{remark}

\subsection{Informal outline}
\label{ssec:outline}

The starting point of our construction is the following observation. Let $X$ be a smooth compact manifold, acted on by a Lie group $G$. By definition, its equivariant homology is given by the homology of its homotopy quotient:
\[
H_*^G(X) = H_*(X\times_{G} EG).
\]

It follows from Gugenheim and May's work \cite{gugenheim1974theory} that this  can be rewritten as the homology of a (derived) tensor product of $C_*(G)$-modules:
\[
H_*^G(X) =  H_* ( C_*(X) \otimes_{C_*(G)} C_*(EG) ).
\]

It is tempting to define equivariant Lagrangian Floer homology by a similar formula. Namely, for a pair of Lagrangians $L_0, L_1$ in a symplectic $G$-manifold $M$, one would like to define
\e\label{eq:naive_def_HF_G}
HF_G(L_0, L_1) =  ( CF(L_0, L_1) \otimes_{C_*(G)} C_*(EG) ).
\e
In order to do so, one needs an action of $C_*(G)$ on $CF(L_0, L_1)$: this is what we aim to construct.

Consider Morse homology first: pick Morse functions $f\colon G\to \rr$ and  $h\colon X\to \rr$. From the homotopy transfer theorem we know that the dg algebra and module structures of $C_*(G)$ and $C_*(X)$ induce respectively \Ainf -algebra and \Ainf -module structures on the Morse complexes $CM(G,f)$ and $CM(X,h)$. Though, it is instructive to construct such structures explicitly, in order to transpose them to the Floer setting. 

At the homology level, these become the algebra and module structures induces by the multiplications and action maps
\ea
m_G &\colon G\times G \to G , \\
m_X &\colon G\times X \to X .
\ea
Therefore, to define the \Ainf -operations
\ea
\mu^2_G &\colon CM(G,f)\otimes CM(G,f) \to CM(G,f) , \\
\mu^{1|1}_X &\colon CM(G,f)\otimes CM(X,h) \to CM(X,h),
\ea
a natural choice to consider are the Morse chain pushforwards of $m_G$ and $m_X$ \cite[Sec.~2.8]{KMbook}, i.e. by counting grafted lines as in Figure~\ref{fig:grafted_lines_mult_Y} (see also \cite{audin2014morse}, and \cite[Sec.~3.2]{equiv} for the grafted line point of view). 
\begin{figure}[!h]
    \centering
    \def\svgwidth{\textwidth}
   \includegraphics[scale=.15]{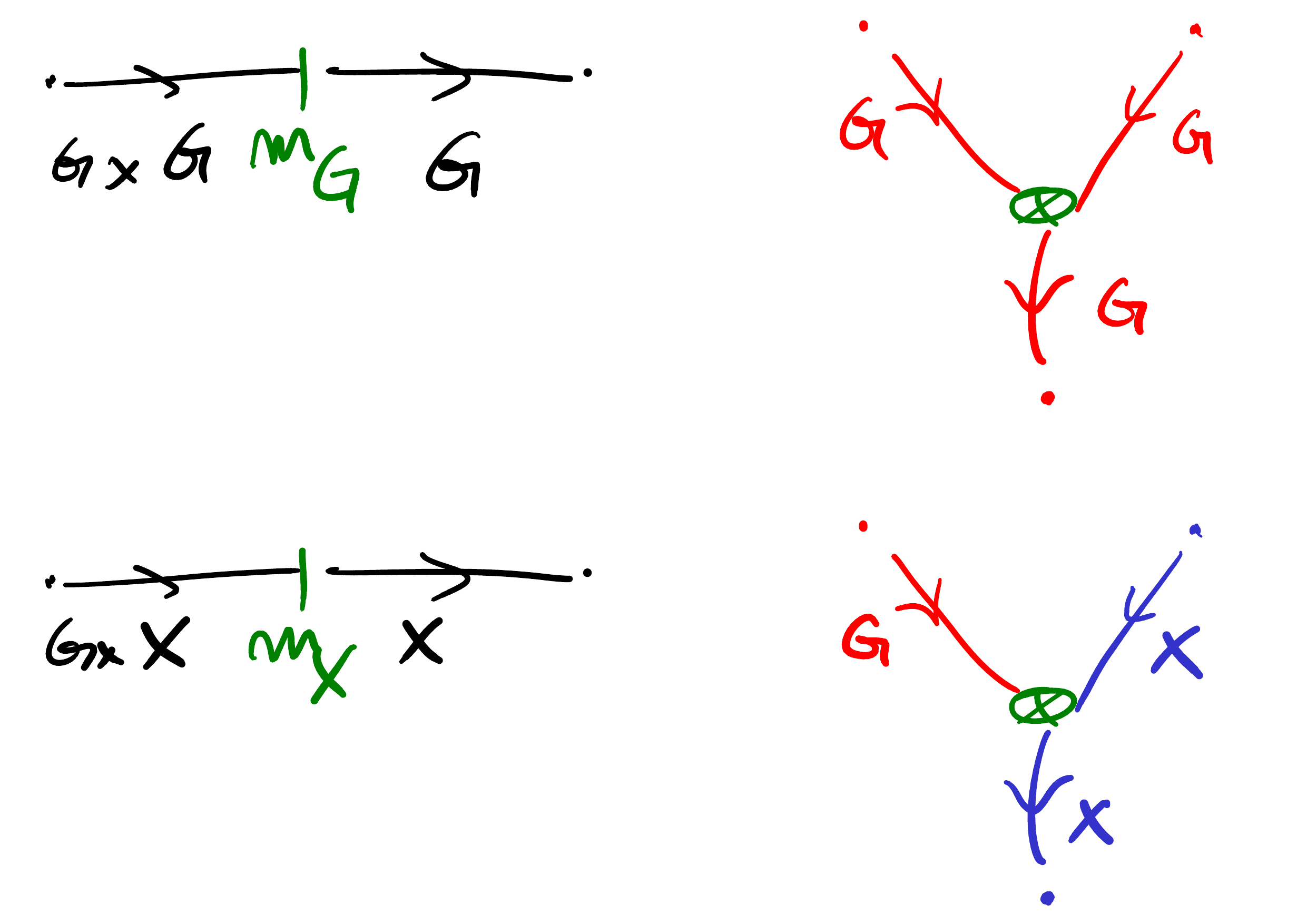}
      \caption{On the left, ``grafted Morse flow lines'' inducing the pushforwards of $m_G$ and $m_X$. On the right, their corresponding ``multiplicative Y's''.}
      \label{fig:grafted_lines_mult_Y}
\end{figure}
By unfolding the components in the product parts, these can alternatively be viewed as ``multiplicative Y's'', as on the right of Figure~\ref{fig:grafted_lines_mult_Y}. For example, the bottom right picture represents a triple of flowlines
\ea
\gamma_1 &\colon (-\infty, 0] \to G , \\
\gamma_2 &\colon (-\infty, 0]  \to X , \\
\gamma_3 &\colon [0, +\infty) \to X ,
\ea
limiting to given input and output critical points at the ends, and satisfying the condition
\e\label{eq:mult_cond}
 \gamma_1(0) \cdot \gamma_2(0) = \gamma_3(0) .
\e
These operations are not strictly associative in general, but they are up to homotopy. And the homotopies are given by counting flow trees with 3 inputs, as in Fukaya's construction \cite{fukaya1993morse}, except that vertices involve multiplicatve conditions (\ref{eq:mult_cond}) as opposed to $\gamma_1(0) = \gamma_2(0) = \gamma_3(0)$, see Figure~\ref{fig:mult_trees_mu_2_1} for $\mu^{2|1}_X \colon CM(G,f)^{\otimes 2}\otimes CM(X,h) \to CM(X,h)$, which measures the defect of associativity of $\mu^{1|1}_X$.
\begin{figure}[!h]
    \centering
    \def\svgwidth{\textwidth}
   \includegraphics[scale=.15]{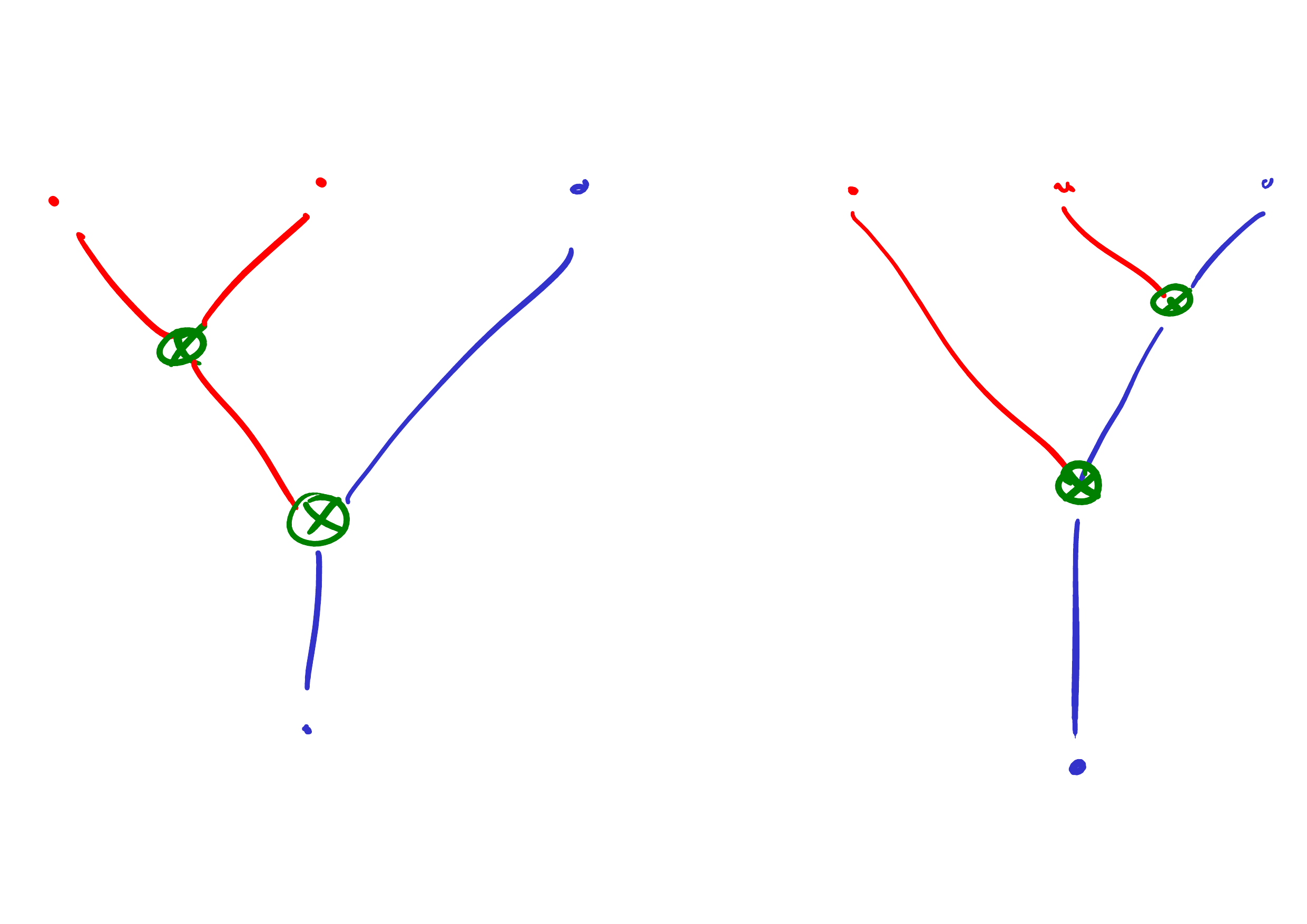}
      \caption{Trees involved in $\mu^{2|1}_X$.}
      \label{fig:mult_trees_mu_2_1}
\end{figure}

More generally, the operations $\mu^{k}_G$ and $\mu^{k|1}_X$ can be defined  analogously, by counting similar trees with respectively $k$ and $k+1$ inputs.

We will construct these operations in Section~\ref{sec:Ain_alg_mod_Morse}, and prove that they satisfy the \Ainf -relations.

Going from Morse to Floer theory, flowlines in $G$ and $X$ respectively become holomorphic strips in $T^*G$ and $T^*X$, grafted lines become quilted strips, and multiplicative trees become ``pseudo-holomorphic foams'' as in Figure~\ref{fig:mult_foam}.
\begin{figure}[!h]
    \centering
    \def\svgwidth{\textwidth}
   \includegraphics[scale=.15]{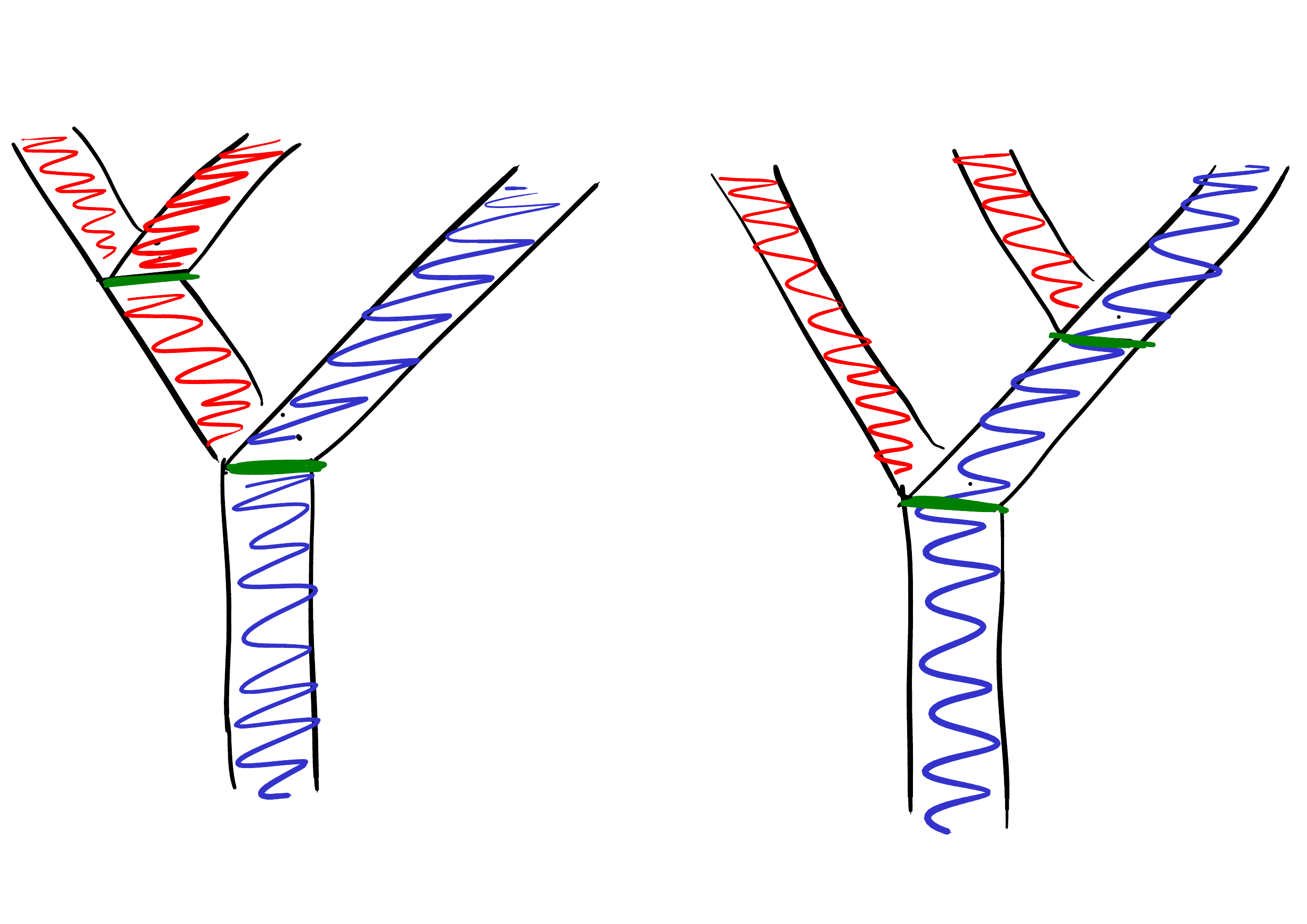}
      \caption{Pseudo-holomorphic foams, Floer-theoretic counterparts of multiplicative trees.}
      \label{fig:mult_foam}
\end{figure}

\begin{remark}\label{rem:foams} Foams are generalizations of quilts \cite{WWquilts}, in quilt language these consist in several ``patches'' (pseudo-holomorphic curves) ``seamed'' together along Lagrangian multi-correspondences (i.e. Lagrangian submanifolds of products of several symplectic manifolds). These can be represented as singular surfaces, and we call them foams in analogy of the ones appearing in \cite{kuperberg1996spiders,khovanov2004sl,kronheimer2019deformation}, though they are very different mathematical objects (in particular they do not correspond to singular surface in a given ambient manifold. Likewise, our multiplicative trees don't look like actual trees in an ambient manifold. 
\end{remark}

Observe here that $T^* X$ can be replaced by mostly any Hamiltonian $G$-manifold $M$ (subject to standard Floer-theoretic assumptions). Indeed, recall that the action and moment maps can be encoded in Weinstein's Lagrangian correspondence
\ea
\Lambda_G(M)   &\subset  T^*G \times M^- \times M \\
&= \left\lbrace ((q,p), m, m') : m' = q m, \ R^*_{q^{-1}} p = \mu(m)  \right\rbrace , \nonumber
\ea
which is the relevant seam condition for the foams we would consider.

Observe finally that if one shrinks the strips in $T^* G$ back to flowlines in $G$ while keeping strips in $M$ (i.e. by taking the Morse function on $G$ smaller and smaller), one ends up with a ``hybrid tree'' as in Figure~\ref{fig:hybrid_tree}. These are collections of flowlines in $G$ and strips in $M$, for which the seam conditions between $G$ and $M$ become 
\e
\gamma(0)\cdot  u_-(0,t) = u_-(0,t)\ ; \ 0\leq t\leq 1.
\e 
\begin{figure}[!h]
    \centering
    \def\svgwidth{\textwidth}
   \includegraphics[scale=.15]{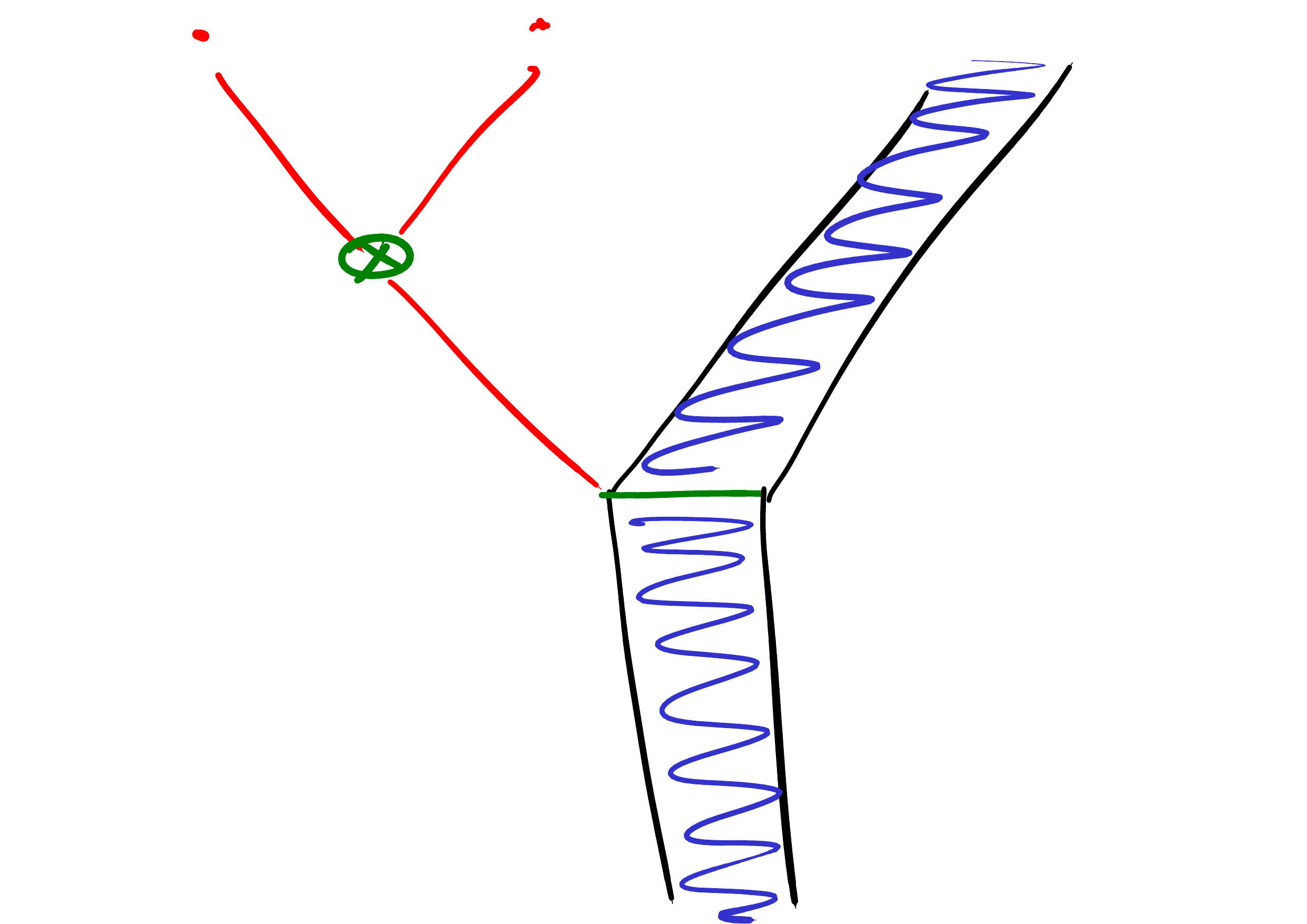}
      \caption{A hybrid tree. The red flowline should touch every point of the green seam, but this is difficult to draw.}
      \label{fig:hybrid_tree}
\end{figure}
The advantage of hybrid trees is that the action need not be Hamiltonian anymore, therefore we will use them instead of foams.

We therefore obtain an \Ainf -module structure on $CF(L_0, L_1)$ over the \Ainf -algebra $CM(G,f)$, and can define versions of equivariant Floer Homology, by replacing formula~(\ref{eq:naive_def_HF_G}) by the appropriate Bar constructions.

\begin{acknow}This project greatly benefitted from conversations with Artem Kotelskiy, Paul Kirk, Mike Miller Eismeier and Wai-Kit Yeung, who declined authorship but deserve some credit. Special thanks to Mike Miller Eismeier for patiently explaining his constructions.

We would also like to thank Paolo Ghiggini for pointing out a subtlety in an earlier attempt to transversality, and Fabian Haiden, Kristen Hendrick,  Dominic Joyce,  Ciprian Manolescu, Thibaut Mazuir,  Alex Ritter,  Sucharit Sarkar and Chris Woodward for helpful conversations.
\end{acknow}

\section{Algebraic constructions}
\label{sec:alg_constr}

Throughout this paper we work over the field $k=\mathbb{F}_2$ for simplicity, and with ungraded complexes, but most of it should hold over more general rings, such as $\zz$, or Novikov rings. Our goal is to construct, starting with an \Ainf -algebra $A$ and a left \Ainf -module $M$, the four complexes 
\e
C^+_A(M),\ C^-_A(M),\ \Ctil^+_A(M),\ C^\infty_A(M)
\e
called respectively the Borel, co-Borel, twisted Borel and Tate complexes.

Applying these constructions to $A=CM(G)$ and $M=CF(L_0, L_1)$ will give four versions of equivariant Lagrangian Floer complexes
\e
CF^{+}_G(L_0, L_1),\ CF^{-}_G(L_0, L_1),\ \widetilde{CF}^{+}_G(L_0, L_1),\ CF^{\infty}_G(L_0, L_1).
\e
We essentially follow Miller Eismeier \cite[Appendix~A]{mike_equiv} with minor adjustments, including:
\begin{itemize}
\item We work in the \Ainf -setting, as opposed to the differential graded one,
\item In \cite{mike_equiv}, Miller Eismeier uses  Bar constructions \emph{reduced} with respect to an augmentation. In our setting, $A$ is not strictly unital (though it should be homotopy unital), therefore we work with \emph{unreduced} constructions, which should be equivalent.
%\item 
%\item 
\end{itemize}

\subsection{The Bar construction}
\label{ssec:Bar_constr}

We recall some basic definitions about \Ainf algebras. For readability, we write $A^k$ for $A^{\otimes k}$, and $(a_1, \ldots , a_k)$ for $a_1 \otimes  \cdots \otimes a_k$.

\begin{defi}\label{def:Ainf_alg} An \Ainf -algebra $A$ is a vector space over $k$ with a collection of operations
\e
\mu_A = (\mu_A^1, \mu_A^2, \ldots); \ \ \mu_A^k\colon A^k \to A
\e
satisfying the \Ainf -relations:
\ea\label{eq:Ainf_alg_rel}
&\forall k\geq 1, \ \ \forall a_1, \ldots , a_k \in A, \\
0 =& \sum_{k+1 = k_1+ k_2; 1\leq i\leq k_1} \mu_A^{k_1}(a_1, \ldots , a_{i-1},  \mu_A^{k_2} (a_i, \ldots , a_{i+k_2}), a_{i+k_2+1}, \ldots , a_{k}). \nonumber
\ea
This can be understood graphically as in Figure~\ref{fig:Ainf_alg_rel}.
\end{defi}

\begin{figure}[!h]
    \centering
    \def\svgwidth{\textwidth}
   \includegraphics[scale=.15]{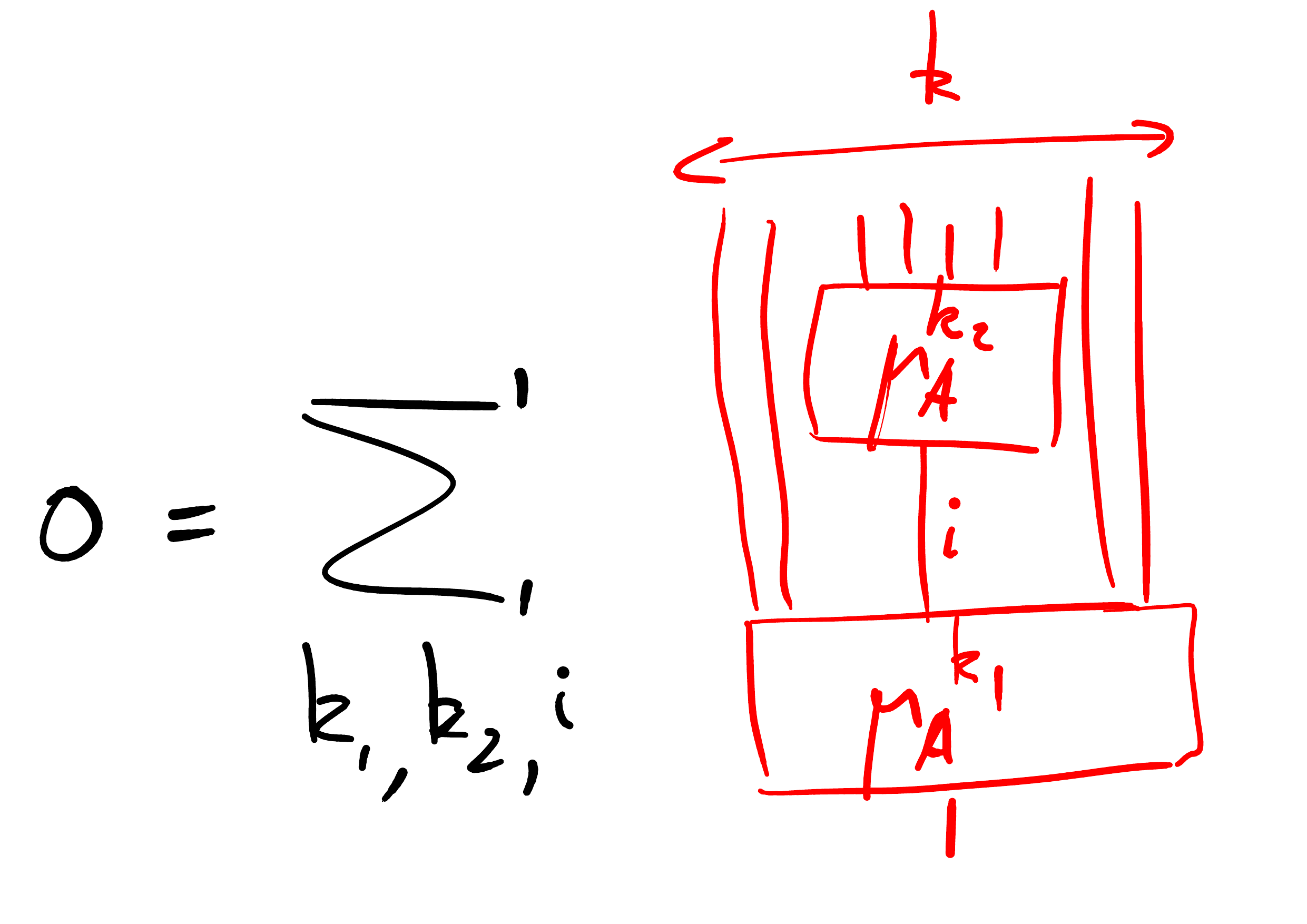}
      \caption{The \Ainf-relations for \Ainf-algebras.}
      \label{fig:Ainf_alg_rel}
\end{figure}

Such a structure can be encoded in a chain complex
\begin{defi}\label{def:Bar_Ainf_alg} Let $A$ be a vector space over $k$. Let $BA$ stand for the tensor coalgebra 
\e
BA = \bigoplus_{k\geq 0}A^k,
\e
with coproduct 
\ea
\Delta\colon BA &\to BA \otimes BA,\\
(a_1, \ldots, a_k) &\mapsto \sum_{i=0}^k (a_1, \ldots, a_i) \otimes (a_{i+1}, \ldots, a_k) .\nonumber
\ea
Any family of maps $(\mu^1, \mu^2, \ldots)$ of the form $\mu^k\colon A^k \to A$, seen as a single map $\mu\colon BA \to A$, uniquely extends as a coalgebra morphism 
\e
\overrightarrow{\mu}\colon BA \to BA.
\e
Furthermore, $\mu$ satisfies the \Ainf -relations if and only if $\overrightarrow{\mu}^2 =0$. In this case, $BA$ is a dg coalgebra.
\end{defi}

\begin{defi}\label{def:Ainf_mod} Let $(A, \mu_A)$ be an \Ainf -algebra. A left \Ainf -module $N$ is a $k$-vector space equipped with a collection of maps
\e
\mu_N = ( \mu_N^{0|1},  \mu_N^{1|1}, \ldots );\ \ \mu_N^{k|1}\colon A^k \otimes N \to N
\e
satisfying  \Ainf -relations similar to those of $A$ (see Figure~\ref{fig:Ainf_mod_rel}):
\ea\label{eq:Ainf_mod_rel}
&\forall k\geq 0, \ \ \forall a_1, \ldots , a_k \in A, n\in N \\
0 =& \sum_{k = k_1+ k_2; 1\leq i\leq k_1 +1} \mu_N^{k_1|1}(a_1, \ldots , a_{i-1},  \mu (a_i, \ldots ),  \ldots , a_k, n), \nonumber
\ea
where the inner $\mu$ either stands for $\mu_N^{k_2|1}$ or $\mu_A^{k_2+1}$, depending on whether $i=k_1+1$ or not.

Likewise, a right \Ainf -module $M$ is a $k$-vector space with a collection of maps $\mu_M^{1|k} \colon M\otimes A^k \to M$ satisfying similar relations.
\end{defi}

\begin{figure}[!h]
    \centering
    \def\svgwidth{\textwidth}
   \includegraphics[scale=.15]{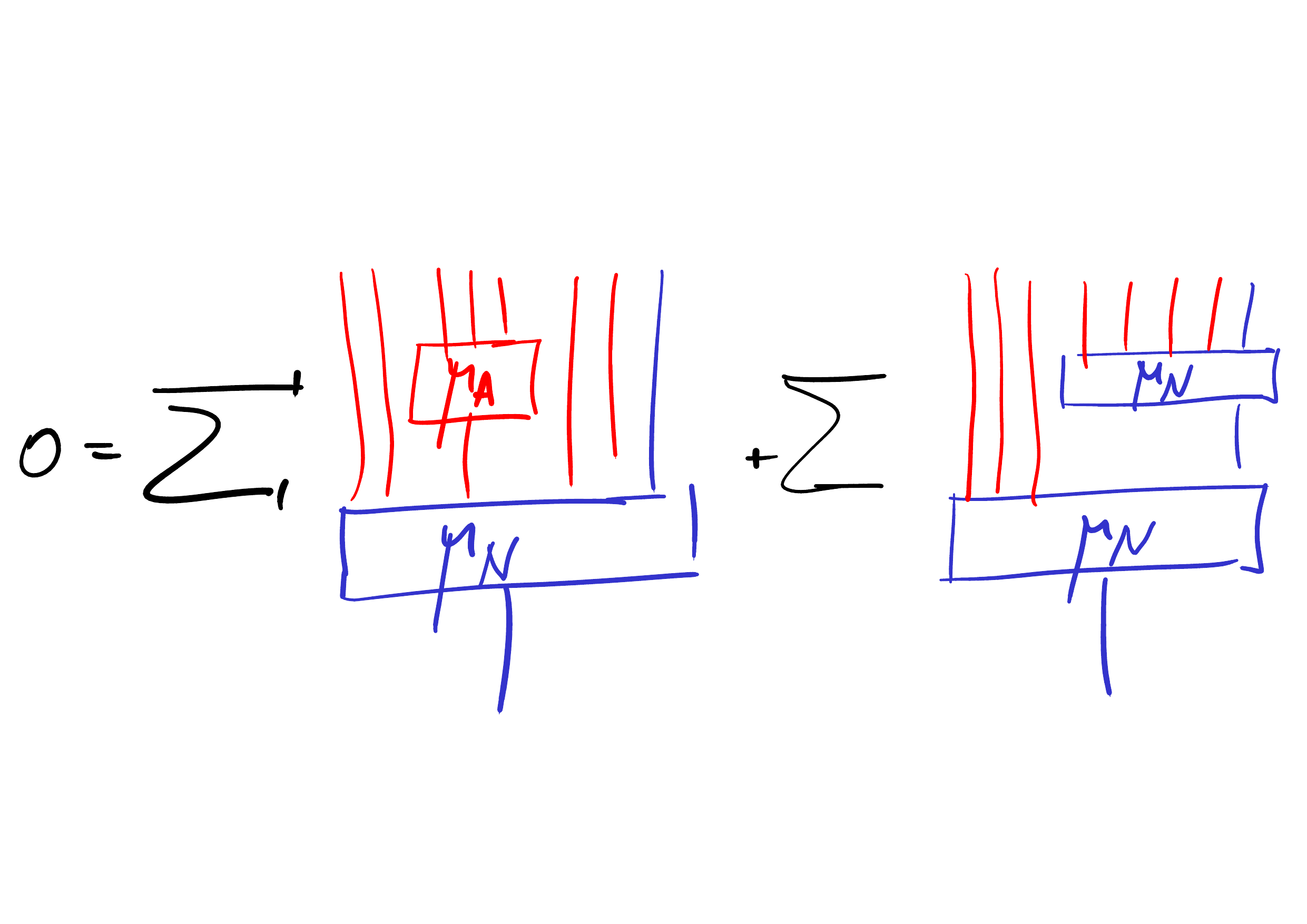}
      \caption{The \Ainf-relations for \Ainf-modules.}
      \label{fig:Ainf_mod_rel}
\end{figure}

These structures can also be encoded in a chain complex.

\begin{defi}\label{def:Bar_Ainf_mod} Let $(C, \mu_A)$ be an \Ainf -algebra, and let $(M, \mu_M)$, $(N, \mu_N)$ be respectively a right and a left \Ainf -module. Let their Bar complex be the chain complex
\e
B(M,A,N) = M\otimes BA \otimes N,
\e
equipped with the differential defined by (see Figure~\ref{fig:bar_diff})
\e
\partial (m, a_1, \ldots , a_k,n) = \sum_{k_1, k_2, k_3} (m, a_1, \ldots, a_{k_1}, \mu( \ldots ), a_{k_1+ k_2+1}, \ldots , a_k,n)
\e
where $k= k_1 + k_2+ k_3$, with
\begin{itemize}
\item $k_1, k_3 \geq -1$, but not both equal to $-1$,
\item $k_2\geq 1$ is the number of inputs of the inner $\mu$, which stands  either for $\mu_M^{1|k_2-1}$, $\mu_A^{k_2}$ or $\mu_N^{k_2-1|1}$.
%\item 
%\item 
\end{itemize}
It follows from the \Ainf -relations that $\partial ^2 = 0$.

The analogous operation of the coproduct of $BA$ is 
\e
\Delta_{B(M,A,N)} \colon B(M,A,N) \to B(M,A,k)\otimes B(k,A,N),
\e
with $k$ standing respectively for the trivial left and right \Ainf\ $A$-module. It follows that $B(M,A,k)$ (resp. $B(k,A,N)$) is  a right (resp. left) dg-comodule over $BA$.
\end{defi}

\begin{figure}[!h]
    \centering
    \def\svgwidth{\textwidth}
   \includegraphics[scale=.15]{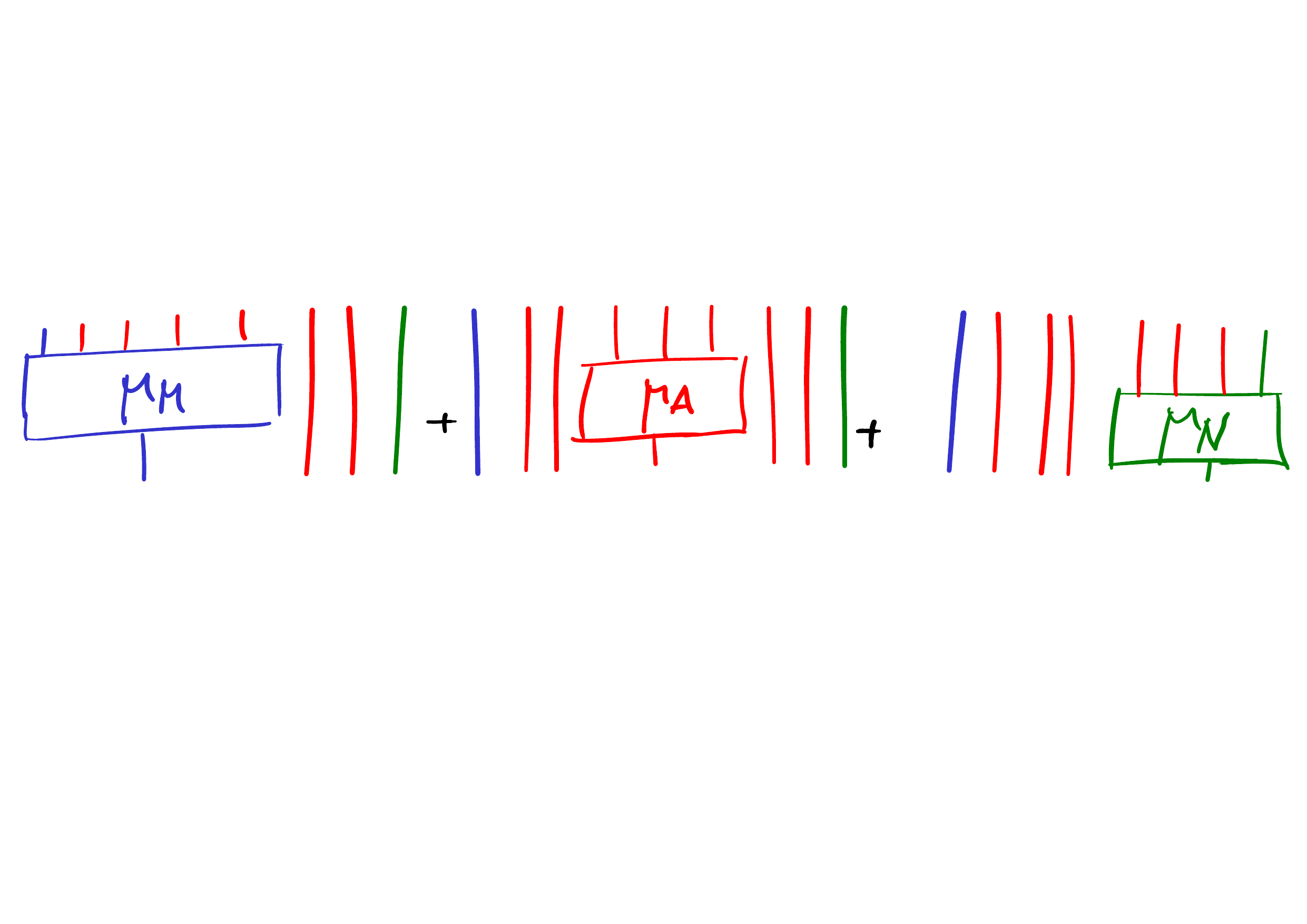}
      \caption{The differential of $B(M,A,N)$.}
      \label{fig:bar_diff}
\end{figure}

We can now define the Borel complex of a left $A$-module $M$ by
\e
C_A^+(M) = B(k,A,M),
\e
which is a left dg-comodule over $BA$.

\subsection{The cobar construction}
\label{ssec:cobar_constr}
The following is the dual notion of an \Ainf -algebra
\begin{defi}\label{def:Ainf_coalg} An \Ainf -coalgebra $A$ is a vector space over $k$ with a collection of operations
\e
\delta_C = (\delta_C^1, \delta_C^2, \ldots); \ \ \delta_C^k\colon C \to C^k
\e
satisfying the \Ainf -relation dual to (\ref{eq:Ainf_alg_rel}):
\e\label{eq:Ainf_coalg_rel}
\forall k\geq 1, \ \ %\forall a_1, \ldots , a_k \in A, \\
0 = \sum_{k+1 = k_1+ k_2; 1\leq i\leq k_1} (id_C^{\otimes i-1} \otimes \delta_C^{k_2}  \otimes id_C^{\otimes k_1 - i} )\circ \delta_C^{k_1}. \nonumber
\e
\end{defi}

\begin{exam} If $A$ is an \Ainf -algebra, let $C= A^*= \mathrm{Hom}_k(A,k)$. It inherits an \Ainf -coalgebra structure $\delta_C$ dual to $\mu_A$ defined by
\e
\left\langle \delta_C^k(c), (a_1, \ldots, a_k) \right\rangle = \left\langle c, \mu_A^k(a_k, \ldots, a_1) \right\rangle .
\e
Notice that we identify duals of tensor products following the rule \e
(A \otimes B)^* \simeq B^* \otimes A^*.
\e
\end{exam}

\begin{figure}[!h]
    \centering
    \def\svgwidth{\textwidth}
   \includegraphics[scale=.15]{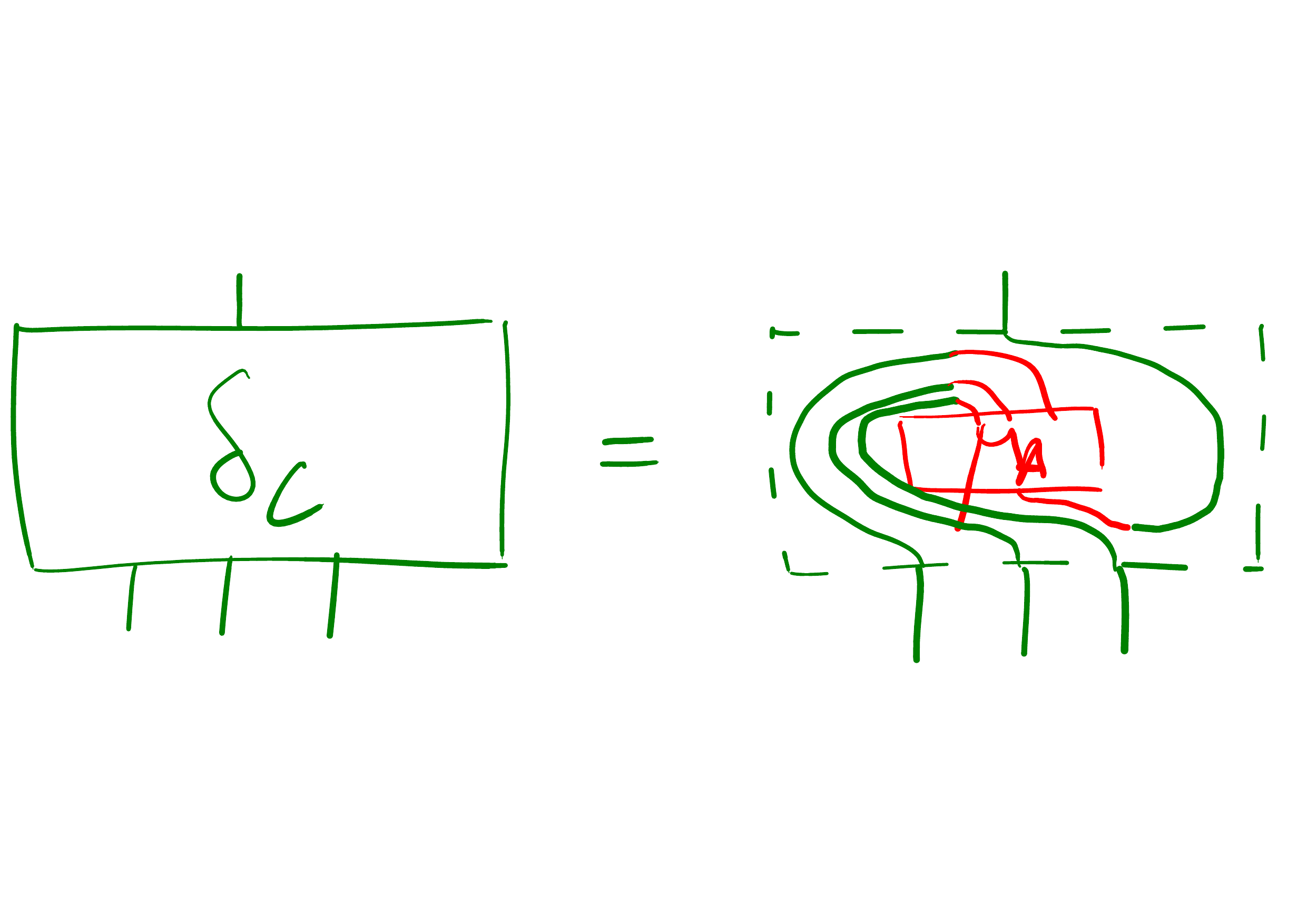}
      \caption{The \Ainf -coalgebra structure on $C=A^*$.}
      \label{fig:dual_Ainf_coalg}
\end{figure}

As for \Ainf -algebras, there is a chain complex associated with an \Ainf -coalgebra.
\begin{defi}\label{def:Cobar_Ainf_coalg} %Let $(C, \delta_C)$ be an \Ainf -coalgebra, and let the tensor algebra

Let $C$ be a vector space over $k$. Let $\Omega C$ stand for the algebra 
\e
\Omega C = \prod_{l\geq 0}C^l,
\e
with product 
\ea
\Pi\colon \Omega C \otimes \Omega C  &\to \Omega C ,\\
(c_1, \ldots, c_l) \otimes (c_1 ', \ldots, c_{l '}') &\mapsto  (c_1, \ldots, c_l, c_1 ', \ldots, c_{l '}') .\nonumber
\ea
A family of maps $\delta = (\delta^1,  \ldots)$ of the form $\delta^k\colon C \to C^k$, seen as a single map $\delta\colon C \to \Omega C$, uniquely extends as an algebra morphism 
\e
\overrightarrow{\delta}\colon \Omega C \to \Omega C.
\e
Furthermore, $\delta$ satisfies the \Ainf -relations if and only if $\overrightarrow{\delta}^2 =0$. In this case, $\Omega C$ is a dg algebra.
\end{defi}

\begin{exam} If $C=A^*$ is the dual of an \Ainf -algebra, then $\Omega C \simeq (BA)^*$.
\end{exam}

\begin{defi}\label{def:Ainf_comod} Let $(C, \delta_C)$ be an \Ainf -coalgebra. A left \Ainf -comodule $N$ is a $k$-vector space equipped with a collection of maps
\e
\delta_N = ( \delta_N^{0|1},  \delta_N^{1|1}, \ldots );\ \ \delta_N^{k|1}\colon  N \to C^k \otimes N
\e
satisfying  \Ainf -relations similar to those of $A$:

%\ea\label{eq:Ainf_md_rel}
%&\forall k\geq 0, \ \ \forall a_1, \ldots , a_k \in A, n\in N \\
%0 =& \sum_{k = k_1+ k_2; 1\leq i\leq k_1 +1} \mu_N^{k_1|1}(a_1, \ldots , a_{i-1},  \mu (a_i, \ldots ),  \ldots , a_k, n), \nonumber
%\ea

\e\label{eq:Ainf_comod_rel}
\forall k\geq 1, \ \ %\forall a_1, \ldots , a_k \in A, \\
0 = \sum_{k = k_1+ k_2; 1\leq i\leq k_1 +1} (id_C^{\otimes i-1} \otimes \delta  \otimes id_C^{\otimes k_1 - i-1} \otimes id_N )\circ \delta_N^{k_1|1}, \nonumber
\e
where the inner $\delta$ either stands for $\delta_N^{k_2|1}$ or $\delta_C^{k_2+1}$, depending on whether $i=k_1+1$ or not.

Likewise, a right \Ainf -comodule $M$ is a $k$-vector space with a collection of maps $\delta_M^{1|k} \colon M \to M\otimes C^k$ satisfying similar relations.
\end{defi}

\begin{figure}[!h]
    \centering
    \def\svgwidth{\textwidth}
   \includegraphics[scale=.15]{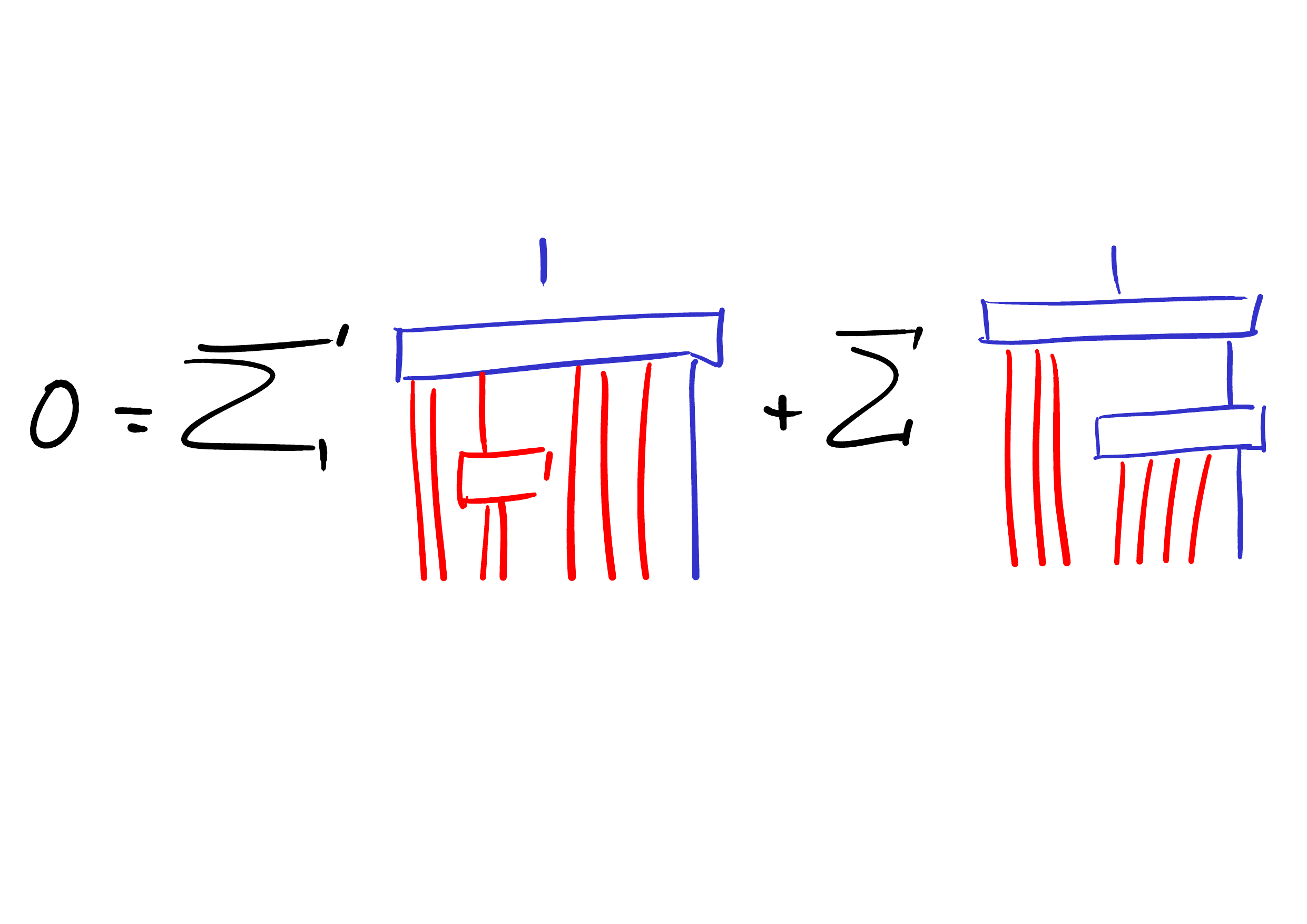}
      \caption{The \Ainf -relations for \Ainf -comodules.}
      \label{fig:Ainf_comod_rel}
\end{figure}

\begin{exam} If $M$ is a left (resp. right) \Ainf -module over $A$, then it is also a left (resp. right) \Ainf -comodule over $C= A^*$.
\end{exam}

As for \Ainf -modules, These structures can also be encoded in a chain complex:

\begin{defi}\label{def:cobar_Ainf_mod} Let $(C, \delta_C)$ be an \Ainf -coalgebra, and let $(M, \delta_M)$, $(N, \delta_N)$ be respectively  right and  left \Ainf -comodules. Let their cobar complex be the chain complex
\e
\Omega (M,C,N) = M\otimes \Omega C \otimes N,
\e
with  differential $\partial = \prod_{l\geq 0}{\partial_l}$ given by
\ea
%\partial (m, c_1, \ldots , c_k,n) &= \prod_{l\geq 0}{\partial_l (m, c_1, \ldots , c_k,n)} \text{, with}\\
\partial_l (m, c_1, \ldots , c_k,n) =& \sum_{0\leq i\leq k+1}{(m, c_1, \ldots , \delta ( c_i ), \ldots , c_k,n) }\\
& \in M\otimes C^l \otimes N \nonumber
\ea
where
\e
 \delta ( c_i ) = \begin{cases}  \delta_M^{1|l-k } ( m ) &\text{ if }i=0  \\
  \delta_C^{l-k+1} ( c_i ) &\text{ if }1\leq i\leq k\\
  \delta_N^{1|l-k } ( n) &\text{ if }i=k+1 
 \end{cases}
\e

It follows from the \Ainf -relations that $\partial ^2 = 0$.

The  operation analogous to the poduct of $\Omega C$ is 
\e
\Pi_{\Omega (M,C,N)} \colon \Omega (M,C,k) \otimes \Omega (k,C,N) \to \Omega (M,C,N),
\e
with $k$ standing respectively for the trivial left and right \Ainf\ $C$-comodule. It follows that $\Omega (M,C,k)$ (resp. $\Omega (k,C,N)$) is  a right (resp. left) dg-module over $\Omega C$.
\end{defi}

\begin{remark}\label{rem:cobar_dual_bar} The construction $\Omega (M,C,N)$ is dual to $B (M,A,N)$ in the sense that 
\e
B (M,A,N)^* \simeq \Omega (N^*,A^*,M^*).
\e
\end{remark}

\begin{remark}\label{rem:cobar_mike} In \cite{mike_equiv}, Miller Eismeier defines his cobar construction as $cB(M,A,N)= \mathrm{Hom}_A(B(M,A,k), N )$ for an \Ainf -algebra $A$ and a pair $M,N$ of right \Ainf -modules. This corresponds to our $\Omega (N^*, A^*, M)$.
\end{remark}

We can now define the co-Borel complex of a left $A$-module $M$ by
\e
C_A^-(M) = \Omega (k,A^*,M),
\e
which, as $C_A^+(M)$,  is a left dg-module over $\Omega A^*$ (and dually a left dg-comodule over $BA$).  Its differential $\partial^- = \partial^-_1 +\partial^-_2$ is drawn in Figure~\ref{fig:coBorel_diff}.

\begin{figure}[!h]
    \centering
    \def\svgwidth{\textwidth}
   \includegraphics[scale=.15]{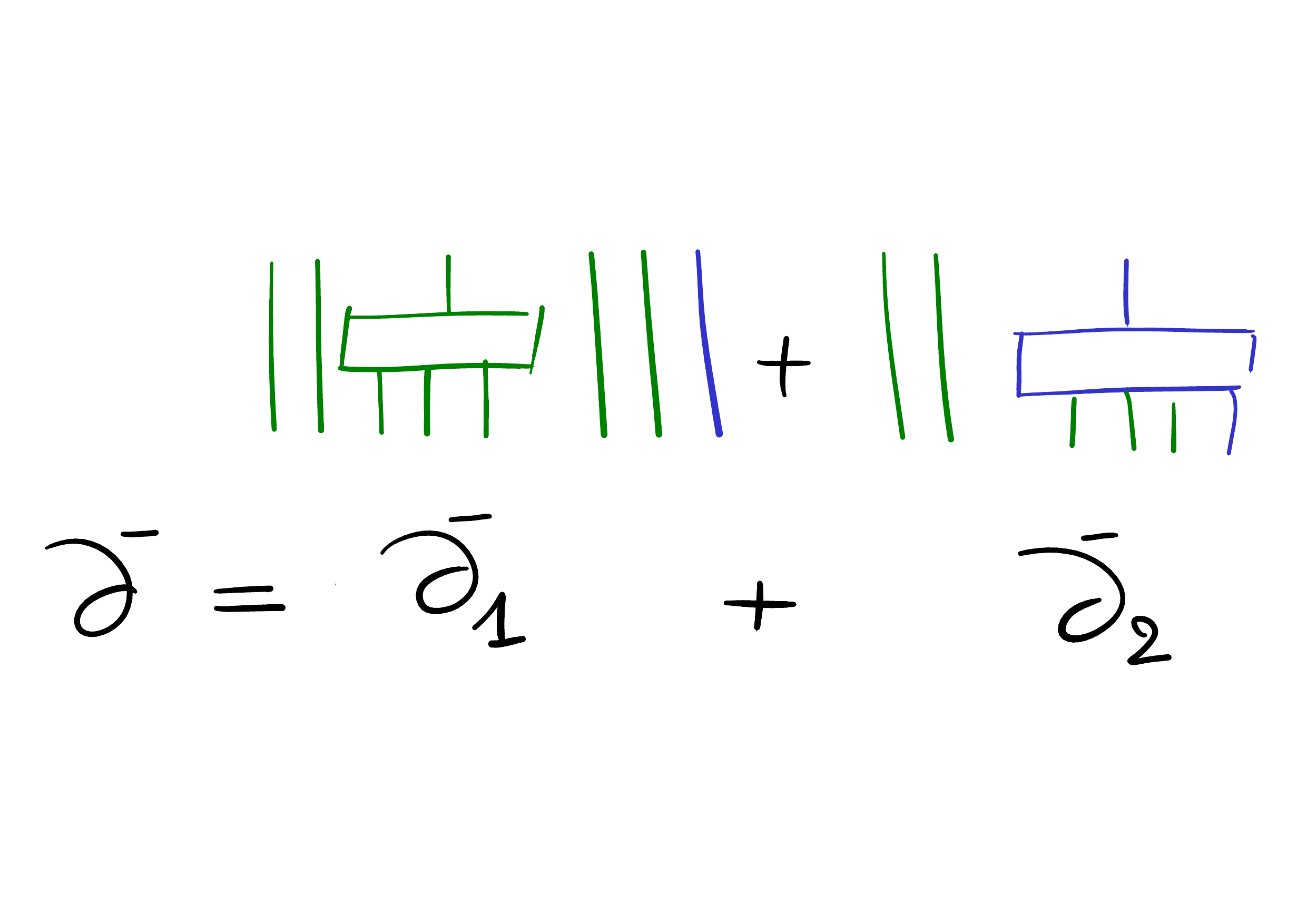}
      \caption{The two contributions to the differential of $C_A^-(M)$.}
      \label{fig:coBorel_diff}
\end{figure}

\subsection{The Tate complex}
\label{ssec:Tate_cpx}

For a $G$-space $X$, equivariant Borel and co-Borel homology respectively correspond to the homology of the homotopy quotient and the homotopy fixed points (which is a spectrum). Removing the word "homotopy", the fixed point set is included in the quotient. So one can consider the cone of this inclusion, and form a third homology group that will make the inclusion fit into a long exact sequence. This is roughly what Tate homology is supposed to be (though this is oversimplifying). It turns out that Tate homology usually enjoys nice properties, which makes the whole package more computable.

The actual homotopy construction is slightly more involved, and defining this inclusion map (the norm map) involves twisting the Borel construction by a dualizing object. To implement this, we follow \cite{mike_equiv} construction, and adapt it to the \Ainf\ setting.

These constructions are best understood in the language of \Ainf\ bimodules. There are several kinds of these (see Figure~\ref{fig:4kinds_bimod}):

\begin{defi}\label{def:Ainf_bimod} Let $A_1, A_2$ be two \Ainf-algebras, and  $C_1, C_2$  two \Ainf-coalgebras.
\begin{itemize}

\item An  $(A_1, A_2)$-bimodule $M$ is a vector space with a collection of operations
\e
\mu_M^{k_1|1|k_2}\colon  A_1^{k_1}\otimes M\otimes A_2^{k_2} \to M
\e
%satisfying the appropriate \Ainf -relations.

\item An  $(C_1, C_2)$-bimodule $M$ is a vector space with a collection of operations
\e
\delta_M^{l_1|1|l_2}\colon  M \to C_1^{l_1}\otimes M\otimes C_2^{l_2}  
\e
%satisfying the appropriate \Ainf -relations.

\item An  $(A_1, C_2)$-bimodule $M$ is a vector space with a collection of operations
\e
\nu_M^{k_1|1|l_2}\colon  A_1^{k_1}\otimes M\to M\otimes C_2^{l_2} 
\e
%satisfying the appropriate \Ainf -relations.

\item An  $(C_1, A_2)$-bimodule $M$ is a vector space with a collection of operations
\e
\chi_M^{l_1|1|k_2}\colon  M\otimes A_2^{k_2} \to C_1^{l_1}\otimes M
\e
\end{itemize}
All these should satisfy the appropriate \Ainf -relations, best explained graphically in Figures~\ref{fig:Ainf_bimod_rel1} and \ref{fig:Ainf_bimod_rel2}.

\begin{figure}[!h]
    \centering
    \def\svgwidth{\textwidth}
   \includegraphics[scale=.15]{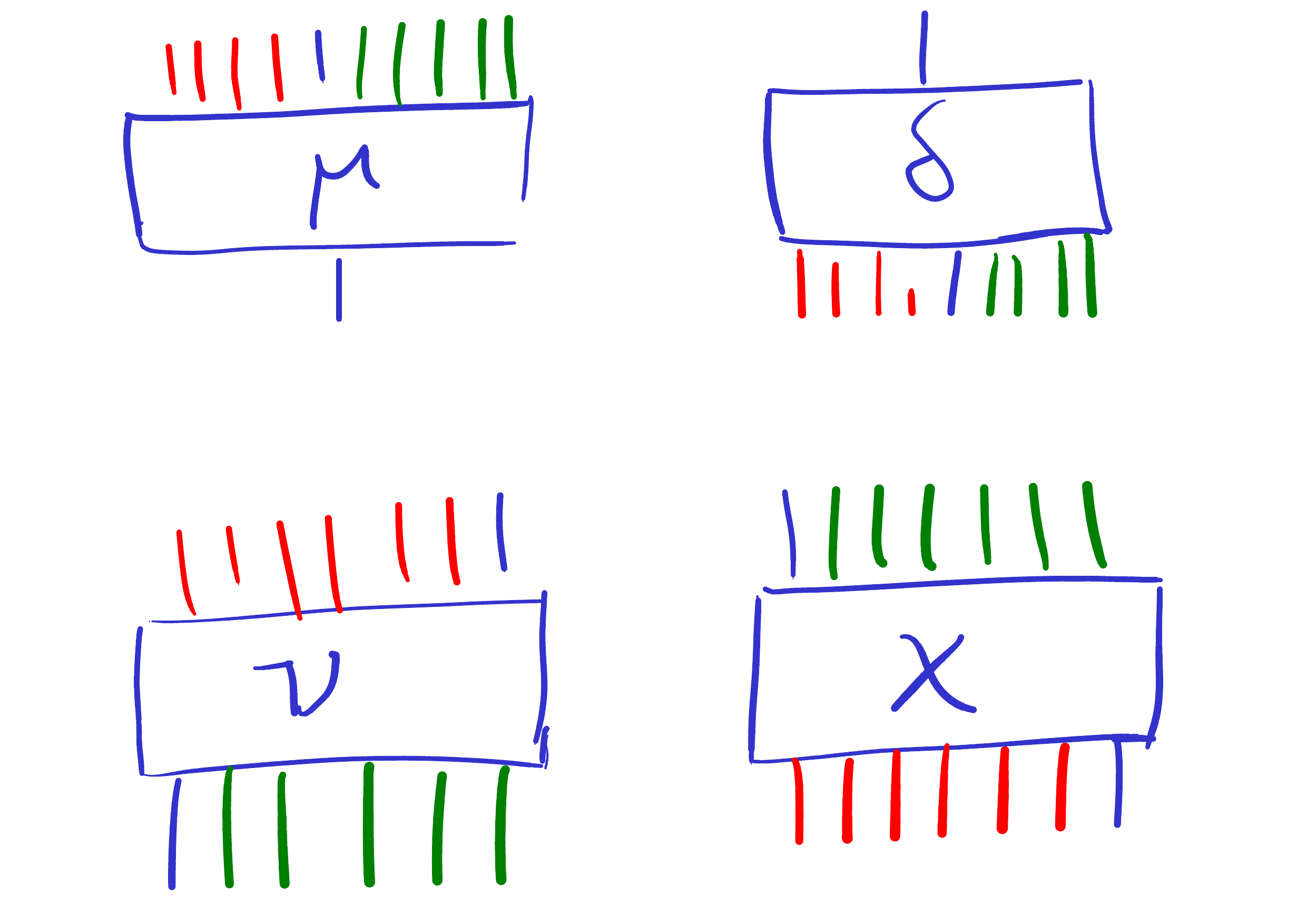}
      \caption{The four kinds of \Ainf-bimodules.}
      \label{fig:4kinds_bimod}
\end{figure}

\begin{figure}[!h]
    \centering
    \def\svgwidth{\textwidth}
   \includegraphics[scale=.15]{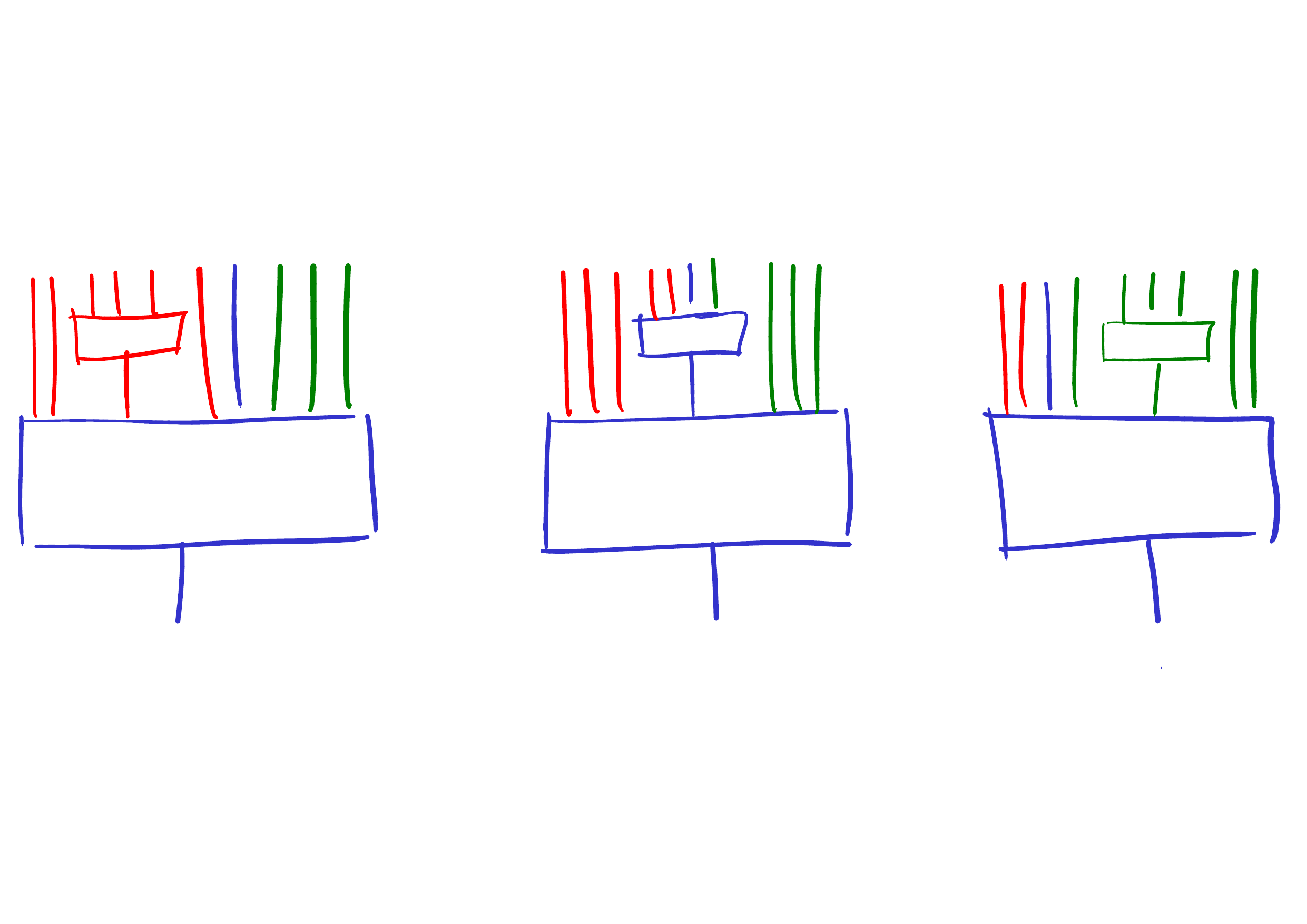}
      \caption{The \Ainf-relations for $(A_1, A_2)$-bimodules.}
      \label{fig:Ainf_bimod_rel1}
\end{figure}

\begin{figure}[!h]
    \centering
    \def\svgwidth{\textwidth}
   \includegraphics[scale=.15]{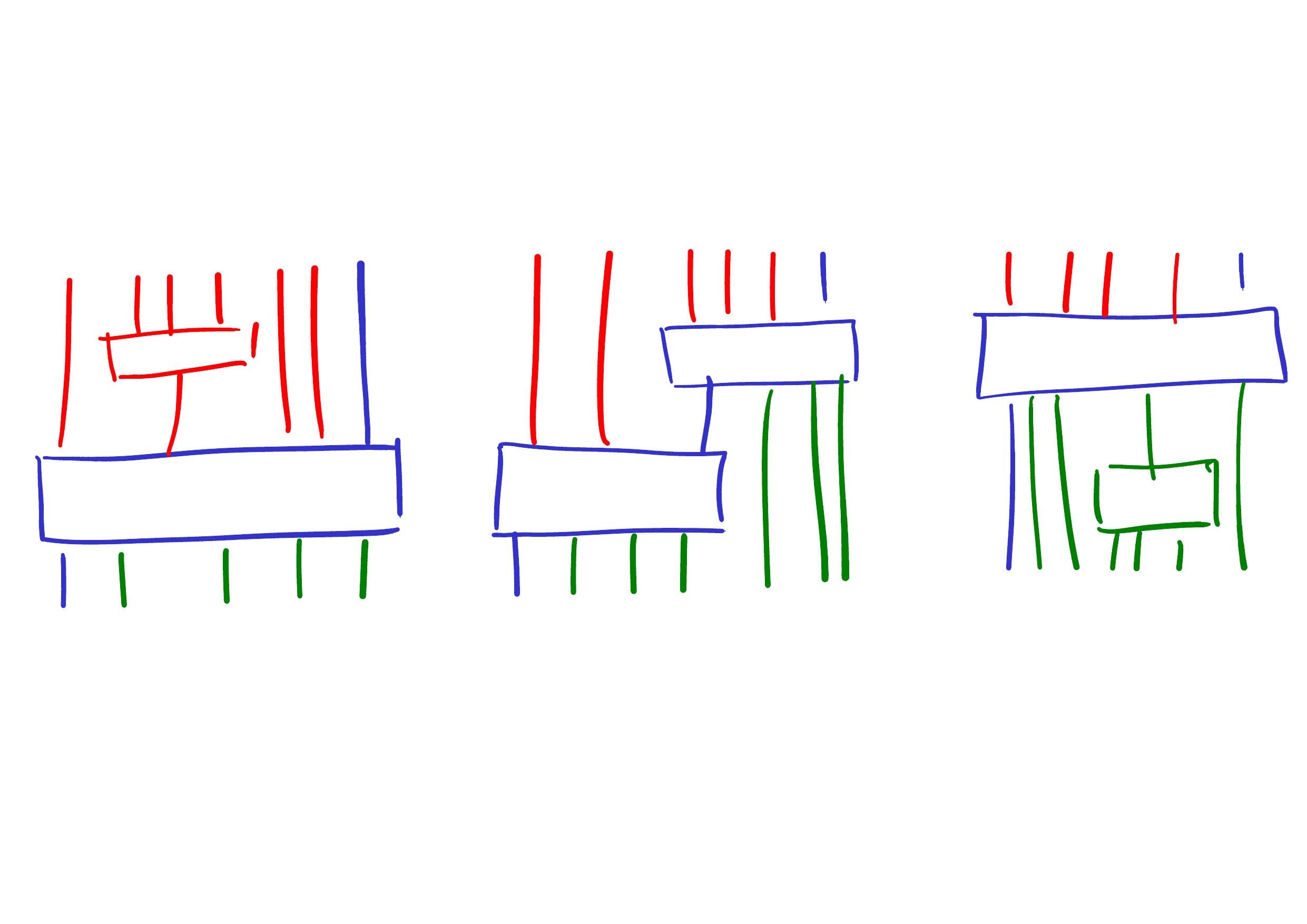}
      \caption{The \Ainf-relations for $(A_1, C_2)$-bimodules.}
      \label{fig:Ainf_bimod_rel2}
\end{figure}

\end{defi}

Following Morita theory,  it is helpful to think about bimodules as morphism between \Ainf-algebras or \Ainf -coalgebras.  The Bar and cobar constructions then give a way to compose them:

\begin{prop}(Composition of \Ainf -bimodules)\label{prop:bar_cobar_compo_bimod} 
Let $G_1,G_2, G_3$ stand for either \Ainf-algebras or \Ainf-coalgebras, let $M$ be a $(G_1, G_2)$-bimodule and $N$ be a $(G_2, G_3)$-bimodule. Then,
\begin{itemize}
\item if $G_2$ is an \Ainf -algebra, then $B(M, G_2, N)$ is a $(G_1, G_3)$-bimodule,
\item if $G_2$ is an \Ainf -coalgebra, then $\Omega(M, G_2, N)$ is a $(G_1, G_3)$-bimodule.
\end{itemize}

\end{prop}

\begin{exam} An \Ainf -algebra $A$ is an $(A,A)$-bimodule.
\end{exam}

\begin{exam} An $(A_1,A_2)$-bimodule is also an $(A_1^*,A_2)$-bimodule and an an $(A_1,A_2^*)$-bimodule.
\end{exam}

From these two observations we get:

\begin{defi}\label{def:dualizing_bimod}Let $A$ be an \Ainf -algebra, its dualizing bimodule $D(A)$ is $A$, seen as an $(A^*,A)$-bimodule.
\end{defi}
This bimodule can be used for twisting the definition of the Borel complex.
\begin{defi}\label{def:twisted_Borel_cpx}Let $M$ be a left \Ainf -module over an \Ainf -algebra $A$, the twisted Borel complex is the dg $\Omega C$-module defined as 
\ea
\Ctil^+_A(M) &= B(D(A), A, M)\\
&= \bigoplus_{k\geq 0}\prod_{l\geq 0}{ C^l \otimes A\otimes A^k\otimes M}.
\ea
In the pictures, we will color strands corresponding  to $C^l$, $A$, $A^k$ and $M$ respectively in green, orange, red and blue.

The differential  of this complex can be decomposed in four contributions
\e
\widetilde{\partial}^+ = \widetilde{\partial}^+_1 +\widetilde{\partial}^+_2 +\widetilde{\partial}^+_3 +\widetilde{\partial}^+_4 , 
\e
as drawn in Figure~\ref{fig:twisted_Borel_diff}.
\end{defi}

\begin{figure}[!h]
    \centering
    \def\svgwidth{\textwidth}
   \includegraphics[scale=.15]{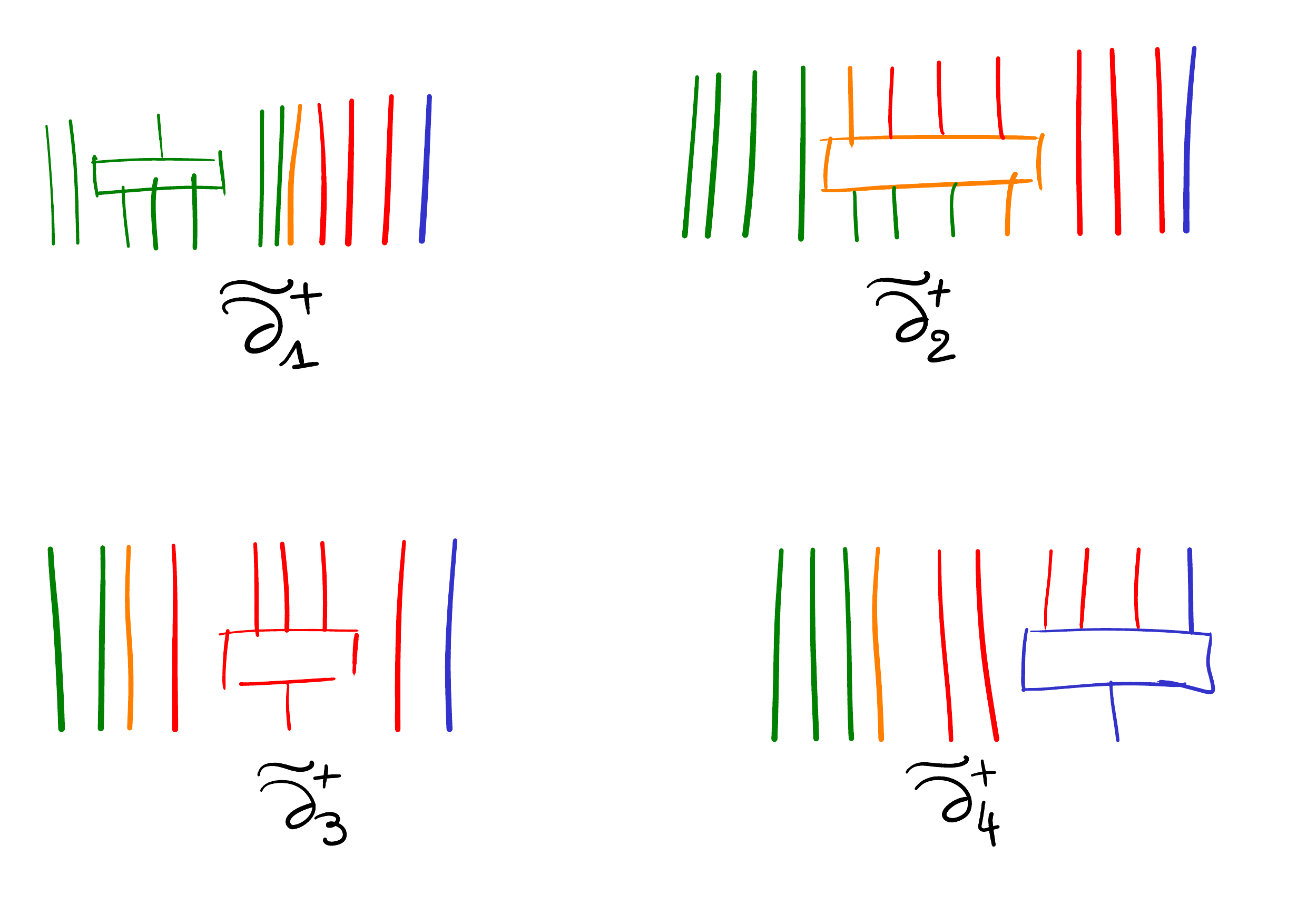}
      \caption{The four contributions to the differential of $\Ctil^+_A(M)$.}
      \label{fig:twisted_Borel_diff}
\end{figure}

To define the fourth complex, consider the following map. 
\begin{defi}\label{def:Norm_map} The norm map
\e
N_M\colon \Ctil^+_A(M) \to C^-_A(M)
\e
is defined by combining all the maps 
\e
N^{k,l}_{l'} \colon  C^l \otimes A\otimes A^k\otimes M \to  C^{l'} \otimes M
\e
drawn in Figure~\ref{fig:Norm_map}.
\end{defi}

\begin{figure}[!h]
    \centering
    \def\svgwidth{\textwidth}
   \includegraphics[scale=.15]{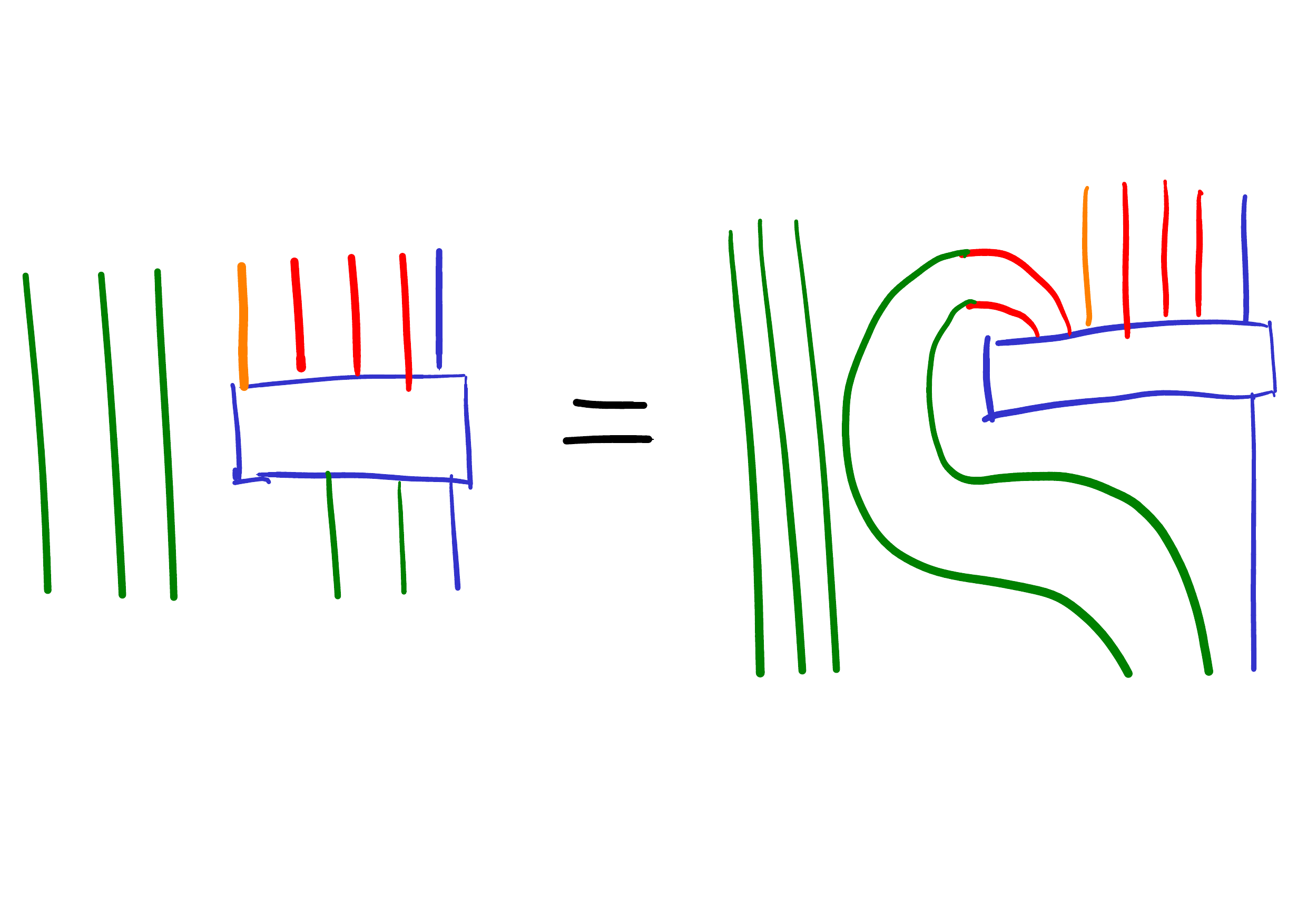}
      \caption{The maps $N^{k,l}_{l'}$ contributing to $N_M$.}
      \label{fig:Norm_map}
\end{figure}

\begin{prop}\label{prop:Norm_commutes_mod_str_and_diff} The norm map $N_M$ is $\Omega C$-equivariant, and commutes with the differentials $\partial^-, \widetilde{\partial}^+$ of $C^-_A(M)$ and $\Ctil^+_A(M)$. 
\end{prop}
\begin{proof}The fact that $N_M$ is $\Omega C$-equivariant is straightforward, and is explained in Figure~\ref{fig:Norm_commutes_mod_str}.

\begin{figure}[!h]
    \centering
    \def\svgwidth{\textwidth}
   \includegraphics[scale=.15]{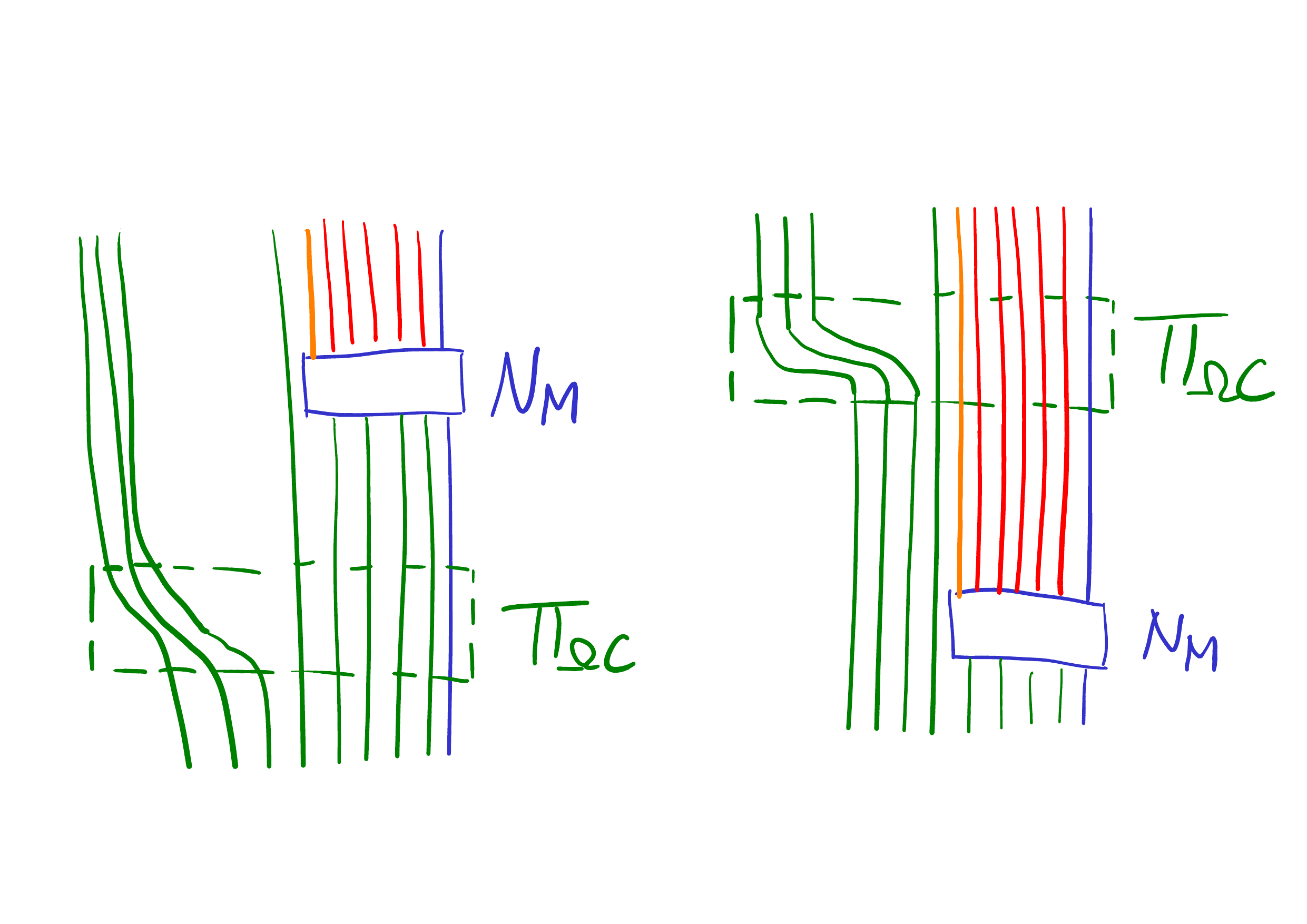}
      \caption{The norm $N_M$ commutes with the $\Omega C$-module structures.}
      \label{fig:Norm_commutes_mod_str}
\end{figure}

We have drawn the contributions to $\partial^- N_M$ and $ N_M \widetilde{\partial}^+$ respectively in Figures~\ref{fig:Norm_comm_diff1} and \ref{fig:Norm_comm_diff2}. Observe that $\partial^-_1 N_M$ has two kinds of contributions, depending on whether the (green) $\delta_C$ operation collides with   $N_M$ or not. If they don't, their positions can be exchanged, and these contributions will cancel in pairs with those of $ N_M \widetilde{\partial}^+_1$. Putting altogether all the other contributions, removing the green vertical lines and dualizing the remaining green strands, one recognizes the \Ainf -relations for $M$. The claim follows.

\begin{figure}[!h]
    \centering
    \def\svgwidth{\textwidth}
   \includegraphics[scale=.15]{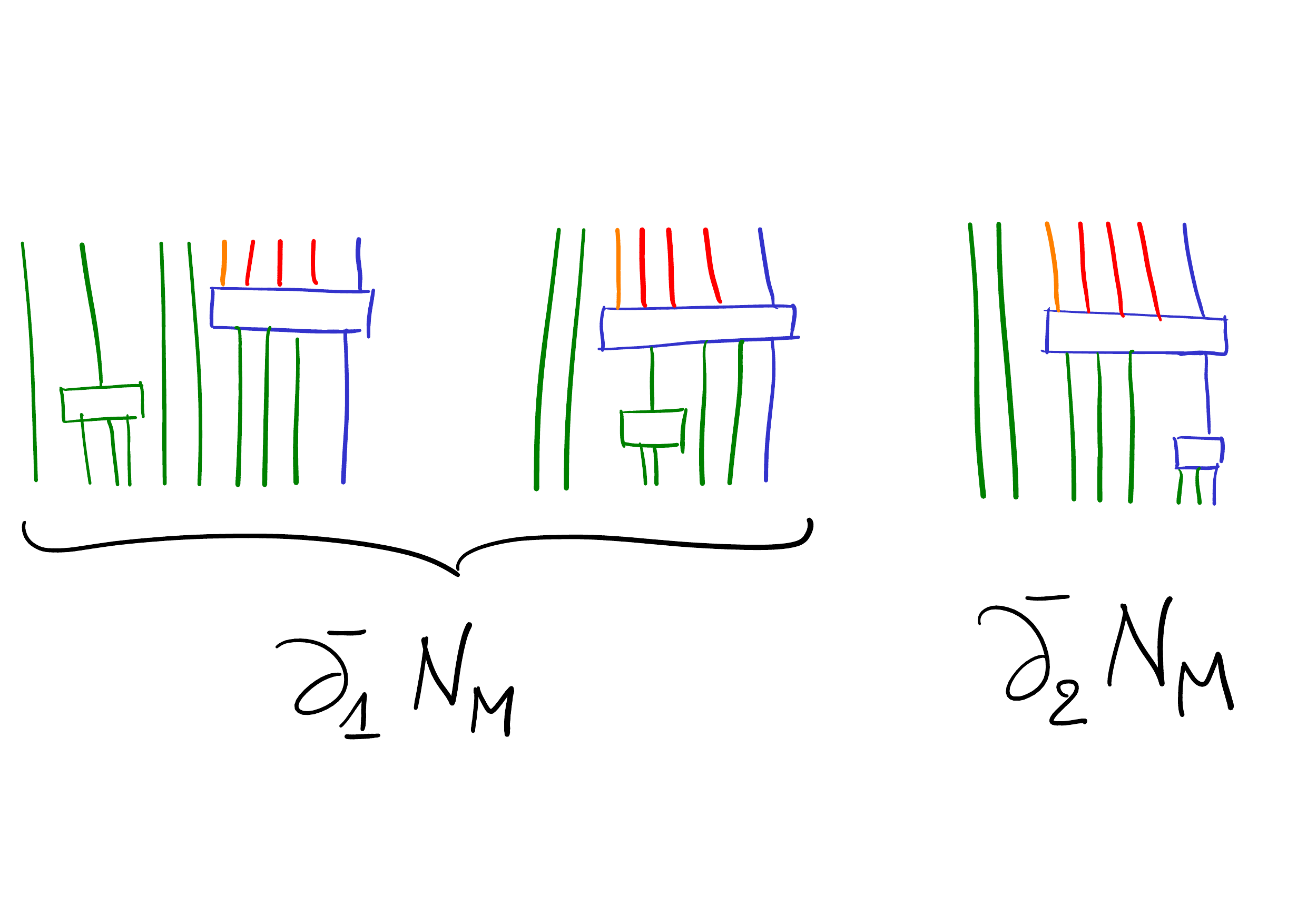}
      \caption{The contributions to $\partial^- N_M$.}
      \label{fig:Norm_comm_diff1}
\end{figure}

\begin{figure}[!h]
    \centering
    \def\svgwidth{\textwidth}
   \includegraphics[scale=.15]{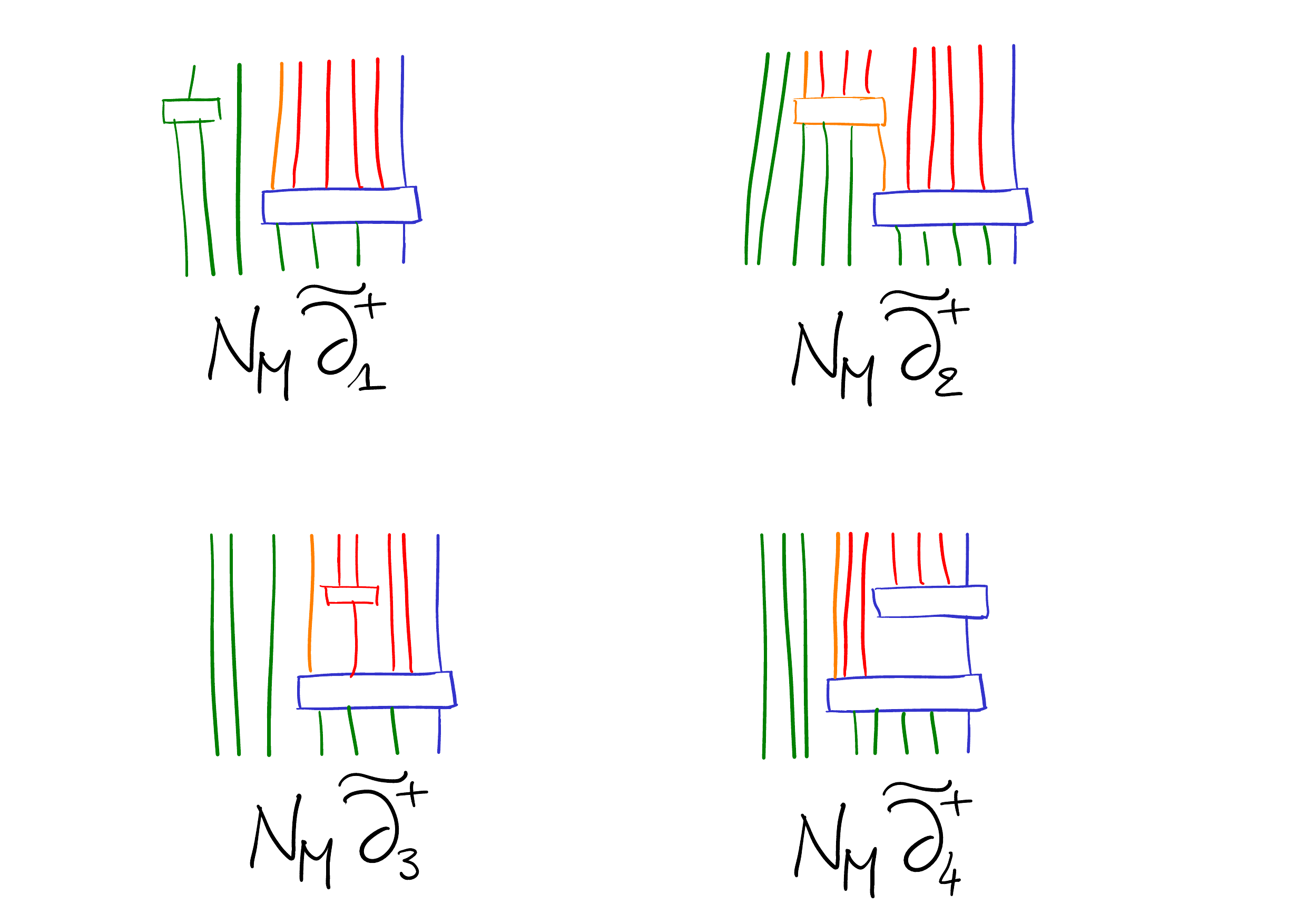}
      \caption{The contributions to $N_M \widetilde{\partial}^+$.}
      \label{fig:Norm_comm_diff2}
\end{figure}

\end{proof}

We can then define the Tate complex as the $\Omega C$ dg-module:
\e
C_A^\infty(M) = \mathrm{Cone}(N_M).
\e

%\subsection{}
%\label{ssec:}
%
%\begin{remark}\label{rem:}
%\end{remark}
%
%\begin{defi}\label{def:}
%\end{defi}
%
%\begin{prop}\label{prop:} 
%\end{prop}
%
%\begin{figure}[!h]
%    \centering
%    \def\svgwidth{\textwidth}
%   \includegraphics[scale=.15]{.pdf}
%      \caption{.}
%      \label{fig:}
%\end{figure}

%\begin{figure}[!h]
%    \centering
%    \def\svgwidth{\textwidth}
%   \includegraphics[scale=.15]{.pdf}
%      \caption{.}
%      \label{fig:}
%\end{figure}

%\begin{figure}[!h]
%    \centering
%    \def\svgwidth{\textwidth}
%   \includegraphics[scale=.15]{.pdf}
%      \caption{.}
%      \label{fig:}
%\end{figure}

%\begin{figure}[!h]
%    \centering
%    \def\svgwidth{\textwidth}
%   \includegraphics[scale=.15]{.pdf}
%      \caption{.}
%      \label{fig:}
%\end{figure}

%\begin{figure}[!h]
%    \centering
%    \def\svgwidth{\textwidth}
%   \includegraphics[scale=.15]{.pdf}
%      \caption{.}
%      \label{fig:}
%\end{figure}

%\begin{figure}[!h]
%    \centering
%    \def\svgwidth{\textwidth}
%   \includegraphics[scale=.15]{.pdf}
%      \caption{.}
%      \label{fig:}
%\end{figure}

%\begin{figure}[!h]
%    \centering
%    \def\svgwidth{\textwidth}
%   \includegraphics[scale=.15]{.pdf}
%      \caption{.}
%      \label{fig:}
%\end{figure}

\section{The $A_\infty$-algebra and module structures in Morse theory}
\label{sec:Ain_alg_mod_Morse}

\subsection{Abstract trees}
\label{ssec:abstract_trees} We now recall some definitions on trees, and set our notation and terminology.

\begin{defi}\label{def:tree}By a \emph{tree} $\tau$ we will mean what is usually called a rooted ribbon tree,  consisting in:
\begin{itemize}
\item A finite set of vertices $\mathrm{Vert}(\tau)$,
\item A finite set of edges of three kinds: internal edges, leaves, and (exactly one) root $r_\tau$. Leaves and roots are also referred to as external, or semi-infinite edges.
\e
\mathrm{Edges}(\tau) = \mathrm{IntEd}(\tau)\cup \mathrm{Leaves}(\tau) \cup \mathrm{Root}(\tau).
\e
The set of leaves is supposed to be ordered (corresponding to the ribbon condition).

\item Source and Target maps. Internal edges have a source and a target:
\e
s,t\colon \mathrm{IntEd}(\tau) \to \mathrm{Vert}(\tau).
\e
Leaves have a target and no source:
\e
t\colon \mathrm{Leaves}(\tau) \to \mathrm{Vert}(\tau).
\e
And the root has a source and no target (the root vertex $v_\tau = s(r_\tau)$)
\e
s\colon \mathrm{Root}(\tau) \to \mathrm{Vert}(\tau).
\e
\end{itemize}
These are such that:
\begin{itemize}

\item There are no cycles (i.e. a cyclic sequence of internal edges $e_1, \ldots , e_k$ with $t(e_i) = s(e_{i+1})$ and $t(e_k) = s(e_{1})$).
\item Each vertex is the source of exactly one edge, and the target of at least two edges. We will denote $\mathrm{In}(v)$ and $\mathrm{Out}(v)$ the sets of incoming and outgoing edges. %$\mathrm{In}(v)$ is ordered and 
$\mathrm{Out}(v)$ consists in one element. The arity of $v$ is the number of incoming edges $ar(v) = \# \mathrm{In}(v) \geq 2 $, and the  valency $val(v) = 1 + ar(v)$ the number of adjascent edges. We say that a tree is trivalent if 
$val(v)= 3$ for each $v$.
\item The order on $\mathrm{Leaves}(\tau)$ defines an order on each $\mathrm{In}(v)$ in the following way: if $\gamma$ and $\gamma'$ are two paths going from leaves $l,l'$ to $e,e'\in \mathrm{In}(v)$, then we order $e$ and $e'$ the same way as $l,l'$ (i.e. we want this order to be independent on the choices of the ``ancestors'' $l,l'$ of $e,e'$).
\end{itemize}
\end{defi}

\begin{defi}\label{def:kappa_n}For $n\geq 2$, let $\kappa_n$ denote the set of isomorphism classes of trees with $n$  inputs. Here we use the obvious isomorphism notion: bijections on vertices and edges preserving the order on leaves. This is the same as the set of $n$-bracketings.
\end{defi}

\begin{defi}\label{def:metric_tree} A \emph{metric tree} $T = (\tau, l)$ consists in a tree $\tau$, with a length function on the set of internal edges:
\e
l\colon \mathrm{IntEd}(\tau) \to [0, +\infty)
\e

\end{defi}

We will regard metric trees modulo equivalence.

\begin{defi}\label{def:edge_collapse}
Define the equivalence relation on metric trees generated by either:
\begin{itemize}
\item If $\tau\simeq \tau'$ and under this identification $l\simeq l'$, then we declare $(\tau, l) \sim (\tau, l)$.
\item If $\tau\in \kappa_n$ and $e\in \mathrm{IntEd}(\tau)$, let $\tau(e)$ consist in the tree obtained from $\tau$ by collapsing $e$ (i.e. by removing $e$ and merging $s(e)$ 	and $t(e)$). If $l(e) = 0$, then we declare $T=(\tau,l)$ to be equivalent to $T' = (\tau(e), l')$, with $l'=l$ on $\mathrm{IntEd}(\tau') =  \mathrm{IntEd}(\tau)\setminus \lbrace e\rbrace$.
\end{itemize}
We say that $T$ is \emph{irreducible} if $l(e) > 0$ for every edge. Each metric tree can be put in a unique irreducible form.
\end{defi}

\begin{defi}\label{def:K_n}
The space $K_n$ of equivalence classes of metric trees with $n$ leaves is the (interior of the) associahedron.

If $\tau \in \kappa_n$, let
\e
K_n^{\tau}\simeq [0, +\infty)^{\mathrm{IntEd}(\tau)} \subset K_n
\e 
correspond to metric trees with given type $\tau$. 
%The interiors 
%\e
%K_n^{\tau, \circ}\simeq (0, +\infty)^{\mathrm{IntEd}(\tau)} 
%\e
%form a partition of $K_n$.
These form a stratification of $K_n$, each strata $K_n^{\tau}$ has codimension given by 
\e
\mathrm{codim} \tau := \sum_{v\in \mathrm{Vert}(\tau)}{ val(v) - 3}.
\e
Define $\kappa^k_n\subset \kappa_n$ to consist in trees of codimension $k$. We will mostly be interested in the codimension zero and one strata:
\begin{itemize}
\item  $\mathrm{codim} \tau =0$ corresponds to  trivalent trees, then we will refer to $K_n^{\tau}$ as a chamber.
\item  $\mathrm{codim} \tau =1$ corresponds to exactly one 4-valent vertex, and trivalent remaining vertices, then we will refer to $K_n^{\tau}$ as a wall.
\end{itemize}
\end{defi}

\subsection{Multiplicative trees and the $A_\infty$-algebra structure on $CM_*(G)$}
\label{ssec:mult_trees_alg_str_CM_G}

Let $G$ be a compact Lie group, and fix a  Morse function $f \colon G\to \rr$.

% , and $X$ a closed smooth manifold on which $G$ acts smoothly. 
%\color{red}Denote respectively
%\ea
%m_G &\colon G\times G \to G, \text{ and} \\
%m_X &\colon G\times X \to X
%\ea
%the multiplication and action maps. \color{black} 
%\color{red} We will assume that both $G$ and $X$ are not discrete (the discrete case is easier and can be treated separately  REF DO IT QUICKLY IN A SEPARATE SECTION? STILL NEED TO EXCLUDE CASE?\color{black} ) F
%\e
%f \colon G\to \rr, \; \; \;  h \colon X\to \rr.
%\e

%\color{red} NOT NEEDED ANYMORE
%and assume, up to moving their critical points slightly, that 
%\ea
%m_G(\mathrm{Crit}f\times \mathrm{Crit}f)\cap \mathrm{Crit}f = \emptyset, \\
%m_X(\mathrm{Crit}f\times \mathrm{Crit}h)\cap \mathrm{Crit}h = \emptyset. 
%\ea

%\begin{lemma}Such Morse functions exist as long as $G$ is not discrete.
%\end{lemma}
%\begin{proof}
%Pick a Morse 
%\end{proof}

\color{black}
We will define the moduli spaces involved in the definition of the $A_\infty$-algebra structure on $CM_*(G,f)$.

%$A_\infty$-module structure on $CM_*(X,h)$ (and the $A_\infty$-algebra structure on $CM_*(G,f)$ by setting $(X,h) = (G,f)$).

Let $\mathfrak{X}(G)$ be the space of vector fields on $G$, and $\mathfrak{X}(G,f) \subset \mathfrak{X}(G)$
%\e
%\mathfrak{X}(G,f) \subset \mathfrak{X}(G), %\; \; \; \mathfrak{X}(X,h) \subset \mathfrak{X}(X)
%\e
the spaces of \emph{pseudo-gradients} for $f$, i.e. those $V\in \mathfrak{X}(G)$ such that $d f . V <0$ outside  $\mathrm{Crit}f$, and such that $V$ is a gradient of $f$ for some metric, in a neighborhood of critical points.% (and likewise for $(X,h)$).

We will consider moduli spaces of trees of gradient flow lines. In order for these to be transversally cut out, we will need to allow the vector fields not to be pseudo-gradients everywhere, as trees could have constant edges at critical points. We will perturb the vector fields near the vertices, as in Abouzaid \cite{abouzaid2011topological}. In order to perturb families consistently, the perturbations will be defined on the ``universal tree'', analogously to Seidel's construction for Fukaya categories \cite[Section~9]{seidel2008fukaya}.

\begin{figure}[!h]
    \centering
    \def\svgwidth{.5\textwidth}
   \includegraphics[scale=.15]{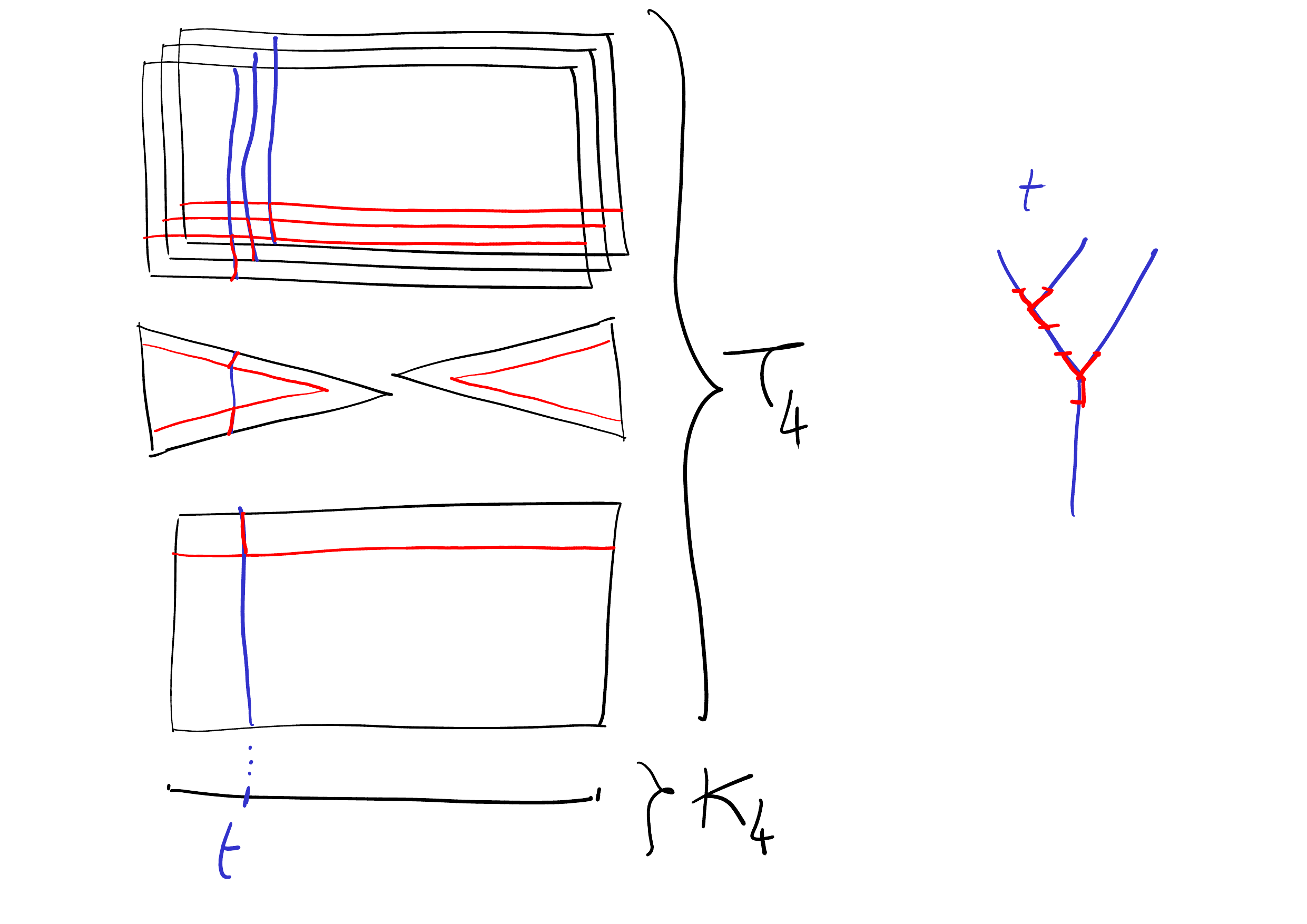}
      \caption{The universal tree $\Tcal_4$.}
      \label{fig:Tcal_4}
\end{figure}

\begin{defi}\label{def:univ_trees} Let $\tau \in \kappa_n$, and $e$ an internal edge. Define the fibered product
\e
\Tcal_e = K_n^\tau \times_{[0, +\infty)} T,
\e
where $K_n^\tau \to [0, +\infty)$ corresponds to the length of $e$, and 
\e
T = \left\lbrace (x,y) \ : \ 0\leq y\leq x \right\rbrace,
\e 
with $T \to [0, +\infty)$ the projection to the first coordinate. The space $\Tcal_e$ ``fibers'' over $K_n^\tau$, and the fiber over $(\tau, l)$ corresponds to the interval $[0, l(e)]$.

If now $e$ is a leaf of $\tau$, let 
\e
\Tcal_e = K_n^\tau \times {(-\infty, 0]} ,
\e
and likewise, if $e$ is a  root,  
\e
\Tcal_e = K_n^\tau \times {[0, +\infty)} .
\e

Define then the \emph{universal tree over $K_n^\tau$} to be
\e
\Tcal_n^{\tau} = \coprod_{e\in \mathrm{Edges}(\tau)}{\Tcal_e} , 
\e
and finally the \emph{universal tree over $K_n$}
\e
\Tcal_n = \left(\coprod_{\tau\in \kappa_n}{\Tcal_n^{\tau}} \right)/\sim , 
\e
where if a metric tree $(\tau,l)$ has an edge $e$ with $l(e)=0$, and $(\tau'=\tau(e), l')$ is its contraction, then we glue the fibre $ \left( \Tcal_n^{\tau '} \right)_{(\tau',l')}$ to $ \left( \Tcal_n^{\tau} \right)_{(\tau,l)}$.

Let also $\Tcal_{n, \leq 1}, \Tcal_{n, \geq 1} \subset \Tcal_n$ stand respectively for the points at distance smaller than $1$ (resp. greater than $1$) from the vertices. 

Let $\Tcal = \coprod_{n\geq 2}{ \Tcal_n}$, with subsets $\Tcal_{ \geq 1}$ and $\Tcal_{ \leq 1}$.  

\end{defi}

\begin{defi}\label{def:univ_pert_pseudograd}

Let the space of \emph{universally perturbed treed pseudogradients} for $(G,f)$

\e
\mathfrak{X}^{\Tcal}(G,f) \subset \mathrm{Maps} \left( \Tcal ,  \mathfrak{X}(G)  \right)  
\e 
consist in smooth maps $V$ that are:
\begin{enumerate}
\item constant on $ \Tcal_{ \geq 1}$ and equal to a pseudo-gradient  $\widehat{V} \in \mathfrak{X}(G,f)$
\item coherent with respect to edge breaking, in the following sense. If $\tau\in \kappa_n$ and $\Ecal \subset \mathrm{IntEd}(\tau)$ is a subset of internal edges splitting $\tau$ into subtrees 
\e
\tau_1\in \kappa_{n_1}, \ldots , \tau_k \in \kappa_{n_k}, 
\e
\end{enumerate}
then gluing metric trees at edges of $\Ecal$ gives gluing maps, for $L\geq 2$:
\e
K_{n_1} \times \cdots \times K_{n_k} \times [L, +\infty)^{\abs{\Ecal}} \to K_n 
\e
Let $\nu_L$ stand for the image of this map, then restricted to this image, there is a map
\e
{\Tcal_{n, \leq 1}}_{|\nu_L} \to \Tcal_{{n_1}, \leq 1} \sqcup \cdots  \sqcup \Tcal_{{n_k}, \leq 1}
\e 
and we want that $V$ factors through this map for $L$ large enough.

$\mathfrak{X}^{\Tcal}(G,f)$ is nonempty (it contains $\mathfrak{X}(G,f)$) and convex, and is equipped with a projection to the constant part:
\ea
\mathfrak{X}^{\Tcal}(G,f) &\to \mathfrak{X}(G,f), \\
 V &\mapsto \widehat{V}.
\ea

\end{defi}

\begin{defi}\label{def:treed_pseudograd_mult_tree} 

Pick $V\in \mathfrak{X}^{\Tcal}(G,f)$, and let $T = (\tau, l)$ be a metric tree. With a slight abuse of notations we will denote the fibre of $\Tcal$ over $T$  by $T$ as well, and we will write it as a union of intervals 
\e
T = \coprod_{e\in \mathrm{Edges}(T)} I_e,
\e 
with
\e 
I_e =  \begin{cases} (-\infty, 0] &\text{ if } e\text{ is a leaf,} \\ [ 0, l(e)] &\text{ if } e\text{ is an internal edge,} \\ [0, +\infty) &\text{ if } e\text{ is the root.} 
\end{cases}
\e
Restricting $V$ to $T$ then gives a map $V_T \colon T \to \mathfrak{X}(G)$.

A \emph{multiplicative tree} in $G$ is a map $\gamma \colon T \to G$ such that:
\begin{itemize}
\item on each edge $I_e$, the restriction $\gamma_e$ of $\gamma$ is a flowline for $V_T$.
\item at every vertex $v$, the multiplicative condition holds.  Assume that the incoming edges at $v$ are given in the following order:
\e
 \mathrm{In}(v) = (e_1, \ldots, e_k) ,
\e
and let $f$ be the outgoing edge. The condition is that 
\e\label{eq:mult_cond_vert_G}
\gamma_{e_1}(v) \times \cdots \times \gamma_{e_k}(v) = \gamma_{f}(v),
\e
where $\gamma_{e}(v)$ stands for $\gamma(t)$, with $t$ the endpoint of $I_e$ corresponding to $v$.

\end{itemize}

%
%
%, define 
%\e
%\mathfrak{X}^{\tau}(G,f;X,h) = \mathfrak{X}(G,f)^{\mathrm{LeftEd}(\tau)}  \times  \mathfrak{X}(X,h)^{\mathrm{TopRightEd}(\tau)}   ,
%\e 
%and pick $V = \lbrace V_e \rbrace_{e\in\mathrm{Edges}(\tau)}$. We will refer to $V$ as a \emph{treed pseudogradient}, or a \emph{$\tau$-pseudogradient}.
%
%A \emph{multiplicative tree} for $V$, denoted $\underline{\gamma}\colon T\to (G,X)$, is a collection of maps
%\ea
%\gamma_e\colon & I_e \to G,\ e\in\mathrm{LeftEd}(\tau), \\
%\gamma_e\colon & I_e \to X,\ e\in\mathrm{TopRightEd}(\tau), 
%\ea
%where $I_e$ is the interval
%\e 
%I_e = [\alpha(e), \beta(e)]= \begin{cases} [-\infty, 0] &\text{ if } e\text{ is a leaf,} \\ [ 0, l(e)] &\text{ if } e\text{ is an internal edge,} \\ [0, +\infty] &\text{ if } e\text{ is the root,} 
%\end{cases}
%\e
%and such that:
%\begin{itemize}
%\item $\gamma_e$ is a flowline for $V_e$: $\gamma_e'(t) = V_e(\gamma_e(t))$.
%\item The following multiplicative condition holds at every vertex $v$: assume that the incoming edges at $v$ are given in the following order:
%\e
% \mathrm{In}(v) = (e_1, \ldots, e_k) ,
%\e
%and let $f$ be the outgoing edge. The condition is that 
%\e
%\gamma_{e_1}(\beta(e_1)) \times \cdots \times \gamma_{e_k}(\beta(e_k)) = \gamma_{f}(\alpha(f)),
%\e
%where $\times$ either stands for $m_G$ or $m_X$.
%\end{itemize}
Since the $\gamma_e$ are flowlines, limits at external ends exist and are critical points. If, as an ordered set, $\mathrm{Leaves}(\tau)= (l_1, \ldots, l_n)$, let
\e
\lim_{-\infty} {\gamma} = (\lim_{-\infty} {\gamma}_{l_1}, \ldots , \lim_{-\infty} {\gamma}_{l_n}) \in (\mathrm{Crit} f)^{n}, % \times \mathrm{Crit} h ,
\e
and let
\e
\lim_{+\infty} {\gamma} = \lim_{+\infty} {\gamma}_{r_\tau} \in \mathrm{Crit} f ,
\e
with $r_\tau$ the root.
\end{defi}

We want to form moduli spaces of trees parametrized by $K_n$: we first define moduli spaces over the chambers $K_n^\tau;\ \tau \in \kappa_n^0$, and glue these along the walls $K_n^\tau;\ \tau \in \kappa_n^1$.

\begin{defi}\label{def:mod_spaces_tau} Let $\underline{x}\in (\mathrm{Crit} f)^{n}$, $y\in \mathrm{Crit} f$, $\tau \in \kappa_n$ and $V \in \mathfrak{X}^{\Tcal}(G,f)$,  define
\e
\Mcal^{\tau}(\underline{x}, y ; V) = \left\lbrace (T,{\gamma})   \ :\ T\in K_n^{\tau}, \ {\gamma}\colon T\to G ,\ \underline{x} = \lim_{-\infty} {\gamma},\ y     = \lim_{+\infty} {\gamma} \right\rbrace,
\e
where ${\gamma}$ is a multiplicative tree for $V$ as above. Let its virtual dimension be defined as:
\e
\mathrm{vdim} ~\Mcal^{\tau}(\underline{x}, y ; {V}) :=  \abs{\underline{x}} - \abs{y} + \mathrm{dim} ~\tau,
\e
where $\abs{y}$ and $\abs{\underline{x}}$ stand for the (sum of) Morse indices.
\end{defi}

\begin{prop}\label{prop:mod_spaces_tau} Given $\tau \in \kappa_n$, there exists a comeagre subset 
\e
\mathfrak{X}^{\Tcal}_{\mathrm{reg}, \tau}(G,f) \subset \mathfrak{X}^{\Tcal}(G,f)
\e 
such that, for $V\in \mathfrak{X}^{\Tcal}_{\mathrm{reg}, \tau}(G,f) $:
\begin{itemize}
\item if $\mathrm{vdim}~ \Mcal^{\tau}(\underline{x}, y ; V) <0$, then $\Mcal^{\tau}(\underline{x}, y ; V)$ is empty,
\item if $\mathrm{vdim}~ \Mcal^{\tau}(\underline{x}, y ; V) =0$, then $\Mcal^{\tau}(\underline{x}, y ; V)$ is a discrete finite set,
\item if $\mathrm{vdim}~ \Mcal^{\tau}(\underline{x}, y ; V) = 1$, then $\Mcal^{\tau}(\underline{x}, y ; V)$ is a smooth 1-manifold with boundary identified to
\e\label{eq:boundary_chamber}
\partial \Mcal^{\tau}(\underline{x}, y ; V) \simeq \bigcup_{e\in \mathrm{IntEd}(\tau)}{\Mcal^{\tau(e)}(\underline{x}, y ; V^{\tau(e)})},
\e
with $V^{\tau(e)}$ obtained from $V$ by removing $V_e$ (recall that $\tau(e)$ consists in the tree obtained from $\tau$ by collapsing $e$).

\end{itemize}

\end{prop}
%MODIF: introduce notation $\gamma_e(v)$ for $\gamma(\alpha(e))$ or $\gamma(\beta(e))$

\begin{proof}
The proof involves a standard transversality argument that we sketch below for the reader's convenience, and refer for example to \cite{audin2014morse} and \cite[Sec.~7]{abouzaid2011topological} for more details.

The universal moduli space $\bigcup_{V}{\Mcal^{\tau}(\underline{x}, y ; V)}$ corresponds to the zero set of a section $s$ of  a Banach bundle 
\e
\Ecal \to \Bcal \times \mathfrak{X}^{\Tcal}(G,f),
\e
where $\Bcal$ is a space of multiplicative trees not subjects to the flowline equation, and $s(\gamma, V) = \dot{\gamma} - V(\gamma)$ is the flowline equation. 

To prove this universal moduli space is smooth, one takes $\eta$
in the cokernel of $D_{(\gamma, V)}s$: for all $(\xi, Y) \in T_{\gamma}\Bcal \times T_{V}\mathfrak{X}^{\Tcal}(G,f)$ one has
\ea
\left\langle \frac{\partial s}{\partial \gamma}. \xi , \eta \right\rangle &= 0 \label{eq:cok1}\\
\left\langle \frac{\partial s}{\partial V}. Y , \eta \right\rangle &= 0 \label{eq:cok2}
\ea
Equation~(\ref{eq:cok1}) implies that $\eta$ satisfies a unique continuation principle, and Equation~(\ref{eq:cok2}) with $Y$ ``bump vector fields'' concentrated at points of $\Tcal_{\leq 1}$ permits to show that vanishes on $\Tcal_{\leq 1}$, which intersect every connected components of the domain of $\eta$, implying $\eta = 0$. Notice that by definition of $\mathfrak{X}^{\Tcal}(G,f)$, $V$ need not be a pseudogradient on $\Tcal_{\leq 1}$. This is what enables one to take such ``bump vector fields'' $Y$, this would not have been possible otherwise, in the case where $\gamma$ has a constant component at a critical point of $f$. From the smoothness of this universal moduli space, one gets a comeagre subset of regular $V$ for which $\Mcal^{\tau}(\underline{x}, y ; V)$ is transversely cut out, by Sard-Smale's theorem.

What differs from Fukaya's Morse flow trees though is the dimension formula, the difference  comes from the fact that we use different vertex conditions: for trivalent vertices our condition involves only one equation ($\gamma_1 (v) \gamma_2 (v) = \gamma_3 (v) $ ) as opposed to two ($\gamma_1 = \gamma_2 = \gamma_3$). Therefore we explain how to obtain it. 

The moduli spaces project to $K_n^\tau$, therefore the virtual dimension is the sum of  $\dim K_n^\tau$ and the dimension of the fibre. Take the fiber over the center of $K_n$, i.e. $(\tau, l)$ with $l\equiv 0$. If $V$ is a constant pseudogradient, the moduli space is identified with the intersection
\e
p( \Ucal_{x_1} \times \cdots \times \Ucal_{x_n}) \cap \Scal_y \subset G,
\e
where $p\colon G^n \to G$ is the product $p(g_1, \ldots, g_n) = g_1 \cdots g_n$ and $\Ucal_{x}, \Scal_{x}$ stand respectively for the unstable and stable manifolds of a critical point $x$. Now generically $p( \Ucal_{x_1} \times \cdots \times \Ucal_{x_n})$ will be of dimension $\abs{\underline{x}}$ and $\Scal_y$ of codimension $\abs{y}$, therefore the expected dimension of the fiber will be $\abs{\underline{x}}-\abs{y}$.

Finally, we discuss the boundary formula (\ref{eq:boundary_chamber}), namely: 
\e\nonumber %\label{eq:boundary_chamber}
\partial \Mcal^{\tau}(\underline{x}, y ; V) \simeq \bigcup_{e\in \mathrm{IntEd}(\tau)}{\Mcal^{\tau(e)}(\underline{x}, y ; V^{\tau(e)})},
\e

Observe first that boundary elements of $\partial \Mcal^{\tau}(\underline{x}, y ; V)$ project to boundary faces of $K_n^\tau$, i.e. with an edge of length $l(e)=0$. The corresponding interval $I_e$ is then reduced to a point, and forgetting the value of $\gamma$ at that point gives the corresponding element of  $\Mcal^{\tau(e)}(\underline{x}, y ; V^{\tau(e)})$. Conversely, given an element of $\Mcal^{\tau(e)}(\underline{x}, y ; V^{\tau(e)})$, the multiplicative condition at vertices (\ref{eq:mult_cond_vert_G}) uniquely determines the value of $\gamma$ at $I_e$.
%MODIF: (somewhere) 
\end{proof}

The boundary formula (\ref{eq:boundary_chamber}) reads as the boundary of a moduli space defined over a chamber of $K_n$ corresponds to the union of the moduli spaces over the boundary walls of that given chamber. Therefore, adjascent chambers can be glued along their common wall\footnote{The key property that allows this is associativity of the multiplication, which we  implicitly use when writing (\ref{eq:mult_cond_vert_G}).}, and when glueing all the chambers together, one obtains a manifold without boundary.

\begin{defi}\label{def:mod_space_over_K_n}

Let

\e
\mathfrak{X}^{\Tcal}_{\mathrm{reg}}(G,f) := \bigcap_{n, \tau\in \kappa_n}{\mathfrak{X}^{\Tcal}_{\mathrm{reg}, \tau}(G,f)} \subset \mathfrak{X}^{\Tcal}(G,f).
\e 
Define for $V \in \mathfrak{X}^{\Tcal}_{\mathrm{reg}}(G,f)$ and $\underline{x} = (x_1, \ldots , x_n),y$ critical points such that $\mathrm{vdim}(\underline{x},y) = \abs{\underline{x}} - \abs{y} + n-2$ either equals zero or one,

\e
\Mcal(\underline{x}, y ; V) = \bigcup_{\tau\in \kappa^0_n}{\Mcal^{\tau}(\underline{x}, y ; V)},
\e
where for $\mathrm{vdim}(\underline{x},y)=0$ this is a disjoint union of finite sets, and for $\mathrm{vdim}(\underline{x},y)=1$ the union is understood as a glueing along boundaries: if $\sigma\in \kappa^1_n$, there is exactly two $\tau, \tau'\in \kappa^0_n$ such that $\sigma = \tau(e) = \tau'(e')$, we then glue  $\Mcal^{\tau}(\underline{x}, y ; V)$ and $\Mcal^{\tau'}(\underline{x}, y ; V)$ along the common part of their boundary $\Mcal^{\sigma}(\underline{x}, y ; V)$.

If $n=1$, let $\Mcal(\underline{x}, y ; V)$ be the space of Morse trajectories for $\widehat{V}$ involved in the Morse differential (i.e. the quotient by $\rr$ of the space of flowlines).

\end{defi}

\begin{prop}\label{prop:compactif_1dim_mod_space} 
If $\mathrm{vdim}(\underline{x},y)=1$, then $\Mcal(\underline{x}, y ; V)$ can be compactified to a compact 1-manifold with boundary $\overline{\Mcal}(\underline{x}, y ; V)$, with boundary given by:
\e 
\partial \overline{\Mcal}(\underline{x}, y ; V)= \bigcup_{z; \ 1\leq i\leq j\leq n} \Mcal(\underline{x}_1, z ; V) \times \Mcal(\underline{x}_2, y ; V),
\e
where (see Figure~\ref{fig:tree_breaking}): $\underline{x}_1 = (x_i, x_{i+1}, \ldots, x_j)$, $z\in\mathrm{Crit}f$ is such that  $\mathrm{vdim}(\underline{x}_1,z)=0$, and $\underline{x}_2 = (x_1, \ldots, x_{i-1}, z,  x_{j+1} \ldots, x_n)$ (and therefore $\mathrm{vdim}(\underline{x}_2,y)=0$).

\end{prop}

\begin{figure}[!h]
    \centering
    \def\svgwidth{\textwidth}
   \includegraphics[scale=.15]{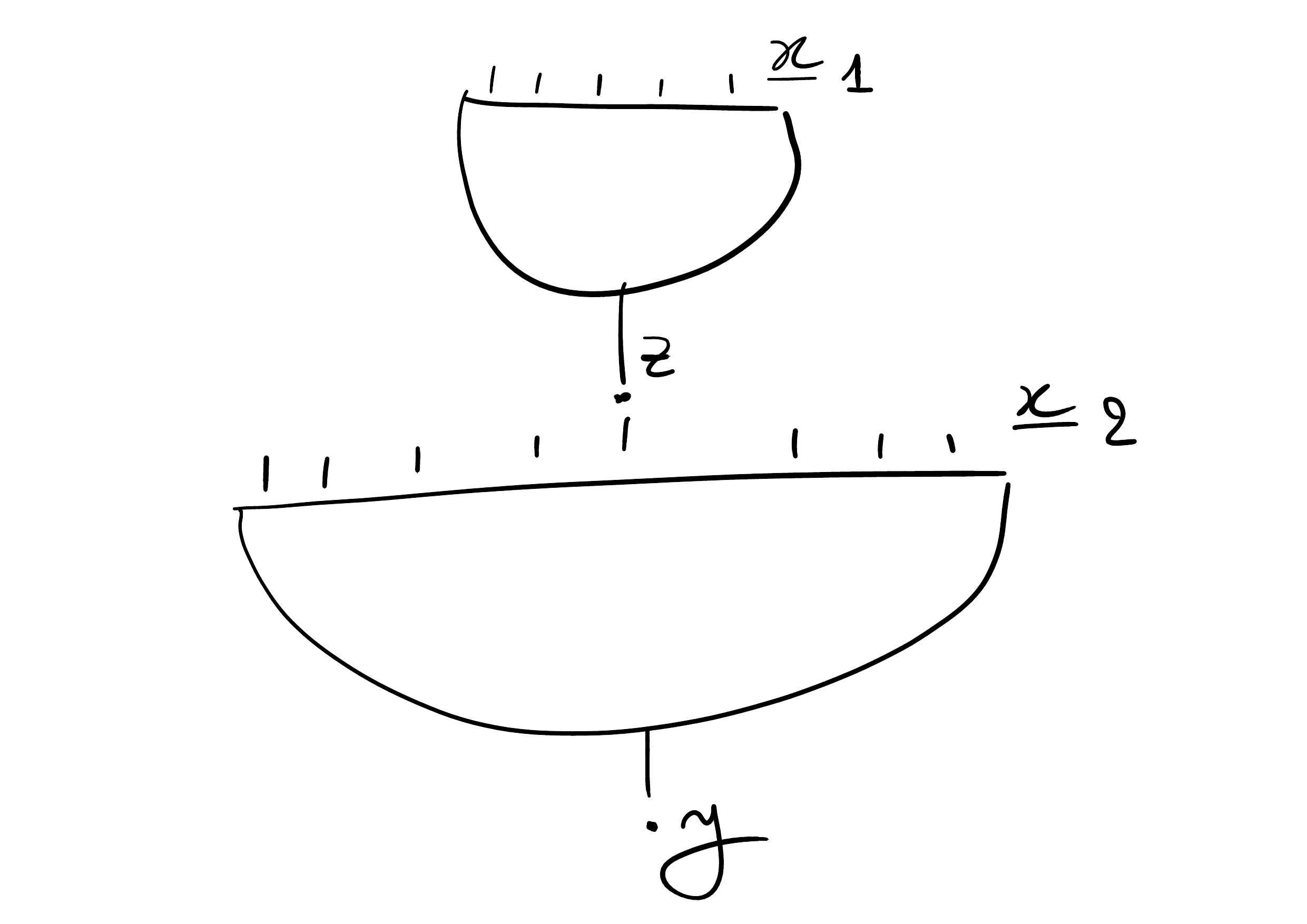}
      \caption{Tree breaking.}
      \label{fig:tree_breaking}
\end{figure}

\begin{proof}These correspond to trees breaking either at internal or external edges.
\end{proof}

We are now ready to define the $A_\infty$-algebra structure on the Morse complex of $G$.

\begin{defi}\label{def:Ainf_str_CM_G}
 Let $A = CM_*(G,f)$, and  $V \in \mathfrak{X}^{\Tcal}_{\mathrm{reg}}(G,f)$. Define for all $n\geq 1$, an operation $\mu_{A}^{n} \colon A^{\otimes n} \to A$ by:
%\e
%\mu_{A}^{n} \colon A^{\otimes n} \to A , 
%\e
%by:
\e
\mu_{A}^{n} (\underline{x}) = \sum_{y;\ \mathrm{vdim}(\underline{x},y)=0} \# \Mcal(\underline{x}, y ; V)\cdot  y
\e
It follows from Proposition~\ref{prop:compactif_1dim_mod_space}  that the $\mu_{A}^{n}$ satisfy the $A_\infty$-relations, therefore $(A, \left\lbrace\mu_{A}^{n} \right\rbrace_n)$ is an $A_\infty$-algebra. 

\end{defi}

\begin{remark}Even though we omit it in the notations, the maps $\mu_{A}^{n}$ depend on the perturbation $V$. One can show, following \cite{mazuir2022higher2}, that different choices of perturbations and Morse functions will yield homotopy equivalent $A_\infty$-algebras.
\end{remark}

\subsection{The $A_\infty$-module structure on $CM_*(X)$}
\label{ssec:Ainf_mod_str_CM_X}

Let now $G$ act on a closed smooth manifold $X$, and $h\colon X \to \rr$ a Morse function. We now want to endow $CM(X,h)$ with an $A_\infty$-module structure over the previously constructed  $A_\infty$-algebra $CM(G,f)$. The construction will be analogous, except that we will be counting multiplicative trees with edges both in $G$ and $X$. We now do the few adjustments in order to do so.

\begin{defi}[Left and top-right parts of a tree, see Figure~\ref{fig:left_topright_tree}]\label{def:LeftEd_TopRightEd}For a tree $\tau\in \kappa_n$, let 
\e
\mathrm{TopRightEd}(\tau)\subset \mathrm{Edges}(\tau)
\e
stand for the maximal leaf (i.e. the maximal element of the ordered set $\mathrm{Leaves}(\tau)$) and all its descendents, all the way down to the root. Let also
\e
\mathrm{LeftEd}(\tau) = \mathrm{Edges}(\tau) \setminus \mathrm{TopRightEd}(\tau).
\e

Let also $\mathrm{LeftEd}(\Tcal)$ and $\mathrm{TopRightEd}( \Tcal)$ stand for the corresponding subsets of the universal tree $\Tcal$.
\end{defi}

\begin{figure}[!h]
    \centering
    \def\svgwidth{\textwidth}
   \includegraphics[scale=.15]{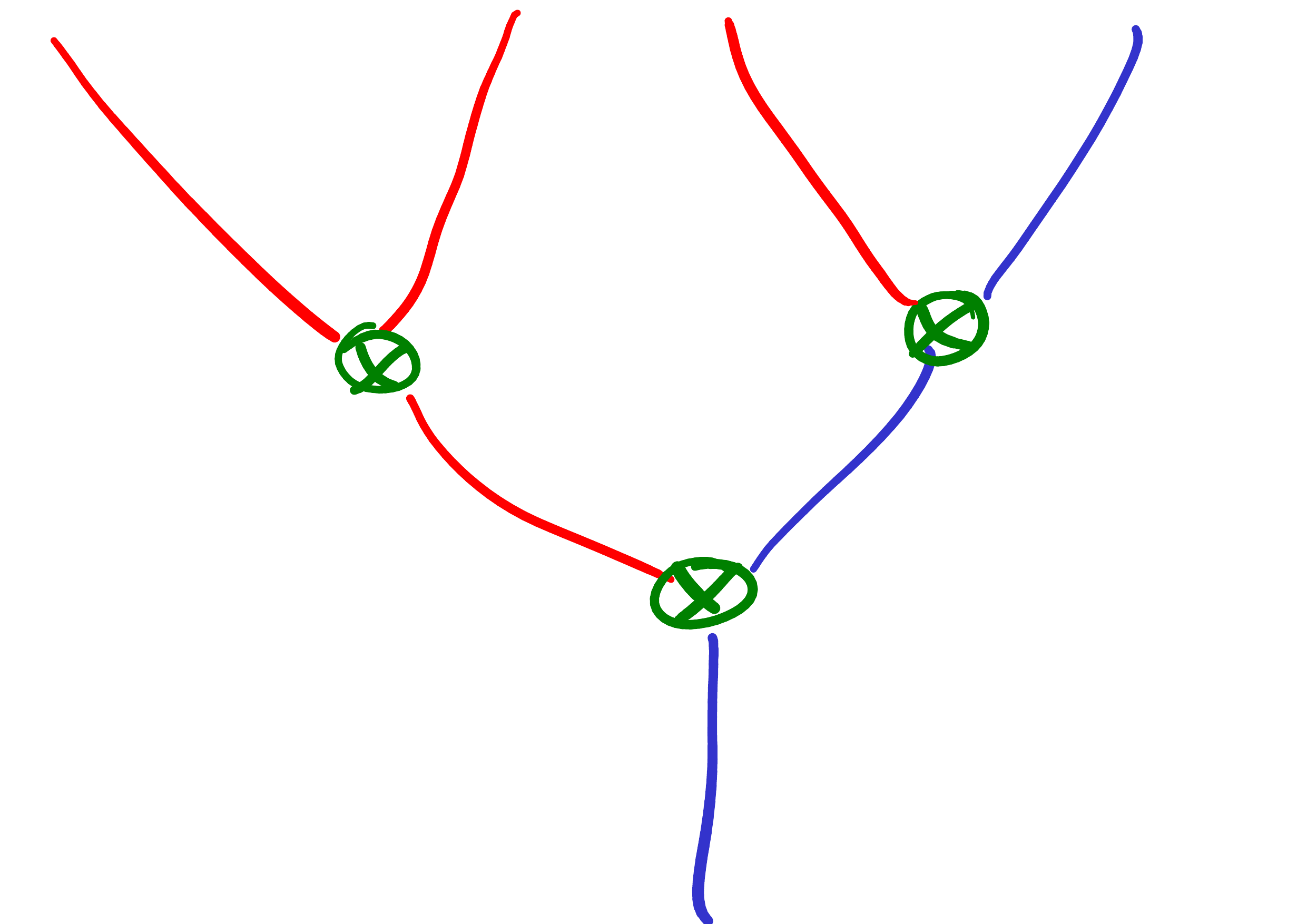}
      \caption{We have drawn a tree $\tau$ with $\mathrm{LeftEd}(\tau)$ in red, and $\mathrm{TopRightEd}(\tau)$ in blue.}
      \label{fig:left_topright_tree}
\end{figure}

The sides  $\mathrm{LeftEd}(\Tcal)$ and $\mathrm{TopRightEd}( \Tcal)$ will be mapped respectively to $G$ and $X$. In order to define the moduli spaces of these corresponding multiplicative trees, let us introduce the relevant space of perturbations, analogous to Definition~\ref{def:univ_pert_pseudograd}.

\begin{defi}\label{def:univ_pert_pseudograd_X}

Fix a $V\in \mathfrak{X}_{\mathrm{reg}}^{\Tcal}(G,f)$, and let
\e
\mathfrak{X}^{\Tcal}(G,f,V;X,h) \subset \mathrm{Maps} \left( \Tcal ,  \mathfrak{X}(G)  \sqcup \mathfrak{X}(X)  \right)  
\e 
consist in smooth maps $W$ that:
\begin{enumerate}
\item map $\mathrm{LeftEd}(\Tcal)$ to $\mathfrak{X}(G)$ and $\mathrm{TopRightEd}( \Tcal)$ to $ \mathfrak{X}(X)$,

\item are locally constant on $ \Tcal_{ \geq 1}$ and equal to either $\widehat{V} \in \mathfrak{X}(G,f)$ or some pseudo-gradient $\widehat{W} \in \mathfrak{X}(X,h)$
\item are coherent with respect to edge breaking, in the following sense. If $\tau\in \kappa_n$ and $\Ecal \subset \mathrm{IntEd}(\tau)$ is a subset of internal edges splitting $\tau$ into subtrees 
\e
\tau_1\in \kappa_{n_1}, \ldots , \tau_k \in \kappa_{n_k}, 
\e
\end{enumerate}
then gluing metric trees at edges of $\Ecal$ gives gluing maps, for $L\geq 2$:
\e
K_{n_1} \times \cdots \times K_{n_k} \times [L, +\infty)^{\abs{\Ecal}} \to K_n 
\e
Let $\nu_L$ stand for the image of this map, then restricted to this image, there is a map
\e
{\Tcal_{n, \leq 1}}_{|\nu_L} \to \Tcal_{{n_1}, \leq 1} \sqcup \cdots  \sqcup \Tcal_{{n_k}, \leq 1}
\e 
and we want that $W$ factors through this map for $L$ large enough. Here, we map $\Tcal_{{n_1}, \leq 1} \sqcup \cdots  \sqcup \Tcal_{{n_k}, \leq 1}$ to $\mathfrak{X}(G)  \sqcup \mathfrak{X}(X)$ by either $W$ or $V$, depending on whether the corresponding subtrees intersect $\mathrm{TopRightEd}( \Tcal)$ or not.

Just as for $\mathfrak{X}^{\Tcal}(G,f)$, the space $\mathfrak{X}^{\Tcal}(G,f,V;X,h)$ is nonempty and convex, and is equipped with a projection to the constant part in $X$:
\ea
\mathfrak{X}^{\Tcal}(G,f,V;X,h) &\to \mathfrak{X}(X,h), \\
 W &\mapsto \widehat{W}.
\ea

\end{defi}

\begin{defi}\label{def:treed_pseudograd_mult_tree_X} 

Pick $W\in \mathfrak{X}^{\Tcal}(G,f,V;X,h)$, and let $T = (\tau, l)$ be a metric tree, written as before $T = \coprod_{e\in \mathrm{Edges}(T)} I_e$.

Restricting $W$ to $T$ then gives a map $W_T \colon T \to \mathfrak{X}(G) \sqcup  \mathfrak{X}(X)$.

A \emph{multiplicative tree} in $(G,X)$ is a map $\gamma \colon T \to G \sqcup X$ such that:
\begin{itemize}
\item it maps  $\mathrm{LeftEd}(T)$ to $G$ and $\mathrm{TopRightEd}( T)$ to $X$,
\item on each edge $I_e$, the restriction $\gamma_e$ of $\gamma$ is a flowline for $W_T$.
\item at every vertex $v$, the multiplicative condition holds.  Assume that the incoming edges at $v$ are given in the following order:
\e
 \mathrm{In}(v) = (e_1, \ldots, e_k) ,
\e
and let $f$ be the outgoing edge. The condition is that 
\e\label{eq:mult_cond_vert_X}
\gamma_{e_1}(v) \times \cdots \times \gamma_{e_k}(v) = \gamma_{f}(v),
\e
where $\gamma_{e}(v)$ stands for $\gamma(t)$, with $t$ the endpoint of $I_e$ corresponding to $v$.  Here, the last $\times$ either stands for the group multiplication $G\times G \to G$ or the action map $G\times X\to X$.
\end{itemize}

As before, if, as an ordered set, $\mathrm{Leaves}(\tau)= (l_1, \ldots, l_n)$, let
\e
\lim_{-\infty} {\gamma} = (\lim_{-\infty} {\gamma}_{l_1}, \ldots , \lim_{-\infty} {\gamma}_{l_n}) \in (\mathrm{Crit} f)^{n-1} \times \mathrm{Crit} h ,
\e
and let
\e
\lim_{+\infty} {\gamma} = \lim_{+\infty} {\gamma}_{r_\tau} \in \mathrm{Crit} h ,
\e
with $r_\tau$ the root.
\end{defi}

Just as for the \Ainf -structure on $CM_*(G)$, for $\underline{x}\in (\mathrm{Crit} f)^{n-1} \times \mathrm{Crit} h$ and $y\in\mathrm{Crit} h$, one can define moduli spaces $\Mcal(\underline{x},y; V,W)$ of multiplicative trees in $(G,X)$ with limits $\underline{x}$ and $y$ at $\pm \infty$. For $W$ in a comeagre subset $\mathfrak{X}^{\Tcal}_{\mathrm{reg}}(G,f,V;X,h)\subset \mathfrak{X}^{\Tcal}(G,f,V;X,h)$, this space is smooth of dimension $\abs{\underline{x}} -\abs{y} +n-2$. When of dimension 0 it is compact, and when of dimension 1, it can be compactified to $\overline{\Mcal}(\underline{x},y; V,W)$, with boundary
\ea
\partial \overline{\Mcal}(\underline{x}, y ; V,W)&= \bigcup_{z\in\mathrm{Crit}h; \ 1\leq i\leq j= n} \Mcal(\underline{x}_1, z ; V,W) \times \Mcal(\underline{x}_2, y ; V,W),\label{eq:boundary_mod_space_module_str} \\
&\cup  \bigcup_{z\in\mathrm{Crit}f; \ 1\leq i\leq j< n} \Mcal(\underline{x}_1, z ; V) \times \Mcal(\underline{x}_2, y ; V,W), \nonumber
\ea
where (see Figure~\ref{fig:tree_breaking}): $\underline{x}_1 = (x_i, x_{i+1}, \ldots, x_j)$, $z$ is such that  $\mathrm{vdim}(\underline{x}_1,z)=0$, and $\underline{x}_2 = (x_1, \ldots, x_{i-1}, z,  x_{j+1} \ldots, x_n)$ (and therefore $\mathrm{vdim}(\underline{x}_2,y)=0$).

Then, with $A= CM_*(G,f)$ and  $M= CM_*(X,h)$,  the zero-dimensional moduli spaces define maps 
\e
\mu_M^{n|1}\colon A^{\otimes n}\otimes M \to M ,
\e
and equation~\ref{eq:boundary_mod_space_module_str} shows that these are \Ainf -module structure maps.

\section{The $A_\infty$-module structure on the Floer complex}
\label{sec:Ainf_mod_str_CF}

\subsection{Geometric setting}
\label{ssec:geom_setting}

We will be working in the setting below, similar with \cite[Hyp.~3.1]{HLSsimplicial} elegant assumptions (contain bot the exact and monotone setting), except that we want our group $G$ compact. 

\begin{assumption}Let:
\begin{itemize}

\item $G$ be a compact Lie group,

\item $(M,\omega)$ be a symplectic manifold, on which $G$ acts by symplectomorphisms,

\item $L_0, L_1\subset M$ be a pair of Lagrangians invariant by the $G$-action.

\end{itemize}
Such that:
\begin{itemize}
\item Any loop of paths from $L_0$ to $L_1$
\e
u\colon (S^1\times [0,1]; S^1\times \left\lbrace 0\right\rbrace , S^1\times \left\lbrace 1\right\rbrace ) \to (M;L_0, L_1)
\e
has zero area and Maslov index.
\item $M$ is either compact or convex at infinity, for some fixed \acs .
\item $L_0$ and $L_1$ are either compact of cylindrical and disjoint at infinity.
\end{itemize}

\end{assumption}

As usual, we will need to use Hamiltonian perturbations and perturb \acs s to achieve transversality, we now define the relevant spaces of perturbations.

\begin{defi}\label{def:} Let $\Hcal_t(M)$ denote the space of smooth maps 
\e
H\colon [0,1] \times M\to \rr
\e
with compact support. For such maps, let $\phi^t_{H}\colon M\to M$ denote their associated Hamiltonian symplectomorphism, i.e. the time $t$ flow of their symplectic gradient.

Let $\Jcal(M)$ stand for the space of $\omega$-compatible \acs s, and $\Jcal_t(M) = C^\infty([0,1], \Jcal(M))$.

Given $L_0, L_1$, a pair $\Fcal = (H_t, J_t)\in \Hcal_t(M)\times \Jcal_t(M)$ such that $\phi_{H_t}^1(L_0)$ and $L_1$ intersect transversely a \emph{Floer datum}.

For such datum, let $\Ical_{H_t}(L_0, L_1)$ (or sometimes $\Ical_{\Fcal}(L_0, L_1)$ ) denote the set of $H_t$-perturbed intersection points, i.e. Hamiltonian chords $x\colon [0,1]\to M$ of the symplectic gradient $X_{H_t}$  with $x(0)\in L_0$ and $x(1)\in L_1$. These are in one to one correspondence with $\phi_{H_t}^1(L_0^N) \cap L_1^N$.

%\e
%\Ical_{H_t}(L_0, L_1):=   \phi_{H_t}^1(L_0) \cap L_1
%\e
%be the finite set of perturbed intersection points (i.e. time 1 Hamiltonian chords from $L_0$ to $L_1$).
\end{defi}

%Let 

Let $Z=\{s+it: 0\leq t \leq 1\}\subset \cc$ be the strip, with its two boundary components $\partial_0 Z= \lbrace t=0\rbrace$, $\partial_1 Z= \lbrace t=1\rbrace$. For $\Fcal = (H_t, J_t)\in \Hcal_t(M)\times \Jcal_t(M)$ and $x, y \in \Ical_{H_t}(L_0, L_1)$, let $\widetilde{\Mcal}(x, y; \Fcal)$ be the moduli space of perturbed $J_t$-holomorphic strips 
\e
u\colon Z\to M
\e
satisfying the \emph{Floer equation}
\e\label{eq:Floer_eq}
\partial_s u + J_t (\partial_t u - X_{H_t})=0,
\e
the Larangian boundary conditions $u_{|\partial_0 Z} \subset L_0$, $u_{|\partial_1 Z} \subset L_1$, and asymptotic to $x$ and $y$ when $s\to-\infty$ and $s\to+\infty$ respectively. Let then $\Mcal(x, y; \Fcal)$ be its quotient by $\rr$ (modulo translations in the $s$-direction).

For $i\in \zz$, let  $\widetilde{\Mcal}(x, y; \Fcal)_i$ and $\Mcal(x, y; \Fcal)_i$ denote the subsets of curves with Maslov index $I(u)=i+1$. 
\begin{prop}\label{prop:def_CF_N}
Assume that $(M;L_0,L_1)$  satisfy the above assumptions. There exists a comeagre subset
\e
\Hcal\Jcal_t^{\mathrm{reg}}(M)\subset \Hcal\Jcal_t(M):=\Hcal_t(M) \times \Jcal_t(M)
\e
of regular peturbations such that, for $\Fcal = (H_t, J_t)\in \Hcal\Jcal_t^{\mathrm{reg}}(M)$, $\phi_{H_t}^1(L_0)$ intersects $L_1$ transversely; $\widetilde{\Mcal}(x, y; \Fcal)_i$ and $\Mcal(x, y; \Fcal)_i$ are smooth of dimension $i+1$ and $i$ respectively. When $i=0$,  $\Mcal(x, y; \Fcal)_0$ is a finite set. In this case, define
\e
CF(M ; L_0,L_1; \Fcal)= \bigoplus_{x\in \Ical(L_0, L_1; H_t)} \Z{2}\, x, 
\e
with differential $\partial = \mu^{0|1}$ defined by 
\e
\partial x = \sum_{y\in \Ical_{\Fcal}(L_0, L_1)} \# \Mcal(x, y; \Fcal)_0 \,y .
\e
%
%Then $\partial^2=0$,  therefore one can define $HF(M ; L_0,L_1; \Fcal)$ as the homology group of this chain complex.

\end{prop}

Now we define the space of perturbations for hybrid trees, analogous to $\mathfrak{X}^{\Tcal}(G,f,V;X,h)$ in Definition~\ref{def:univ_pert_pseudograd_X}.

\begin{defi}\label{def:univ_pert_hybrid_tree}

Fix a $V\in \mathfrak{X}_{\mathrm{reg}}^{\Tcal}(G,f)$, and let
\e
\Pcal ert^{\Tcal}(G,f,V;L_0, L_1) \subset \mathrm{Maps} \left( \Tcal ,  \mathfrak{X}(G)  \sqcup \left( \Hcal\Jcal_t(M)  \right)  \right)  
\e 
consist in smooth maps $P$ that:
\begin{enumerate}
\item map $\mathrm{LeftEd}(\Tcal)$ to $\mathfrak{X}(G)$ and $\mathrm{TopRightEd}( \Tcal)$ to $\Hcal \Jcal_t(M)$,

\item are locally constant on $ \Tcal_{ \geq 1}$ and equal to either $\widehat{V} \in \mathfrak{X}(G,f)$ or some pair $\widehat{P} \in \Hcal\Jcal_t(M)$
\item are coherent with respect to edge breaking, in the following sense. If $\tau\in \kappa_n$ and $\Ecal \subset \mathrm{IntEd}(\tau)$ is a subset of internal edges splitting $\tau$ into subtrees 
\e
\tau_1\in \kappa_{n_1}, \ldots , \tau_k \in \kappa_{n_k}, 
\e
\end{enumerate}
then gluing metric trees at edges of $\Ecal$ gives gluing maps, for $L\geq 2$:
\e
K_{n_1} \times \cdots \times K_{n_k} \times [L, +\infty)^{\abs{\Ecal}} \to K_n 
\e
Let $\nu_L$ stand for the image of this map, then restricted to this image, there is a map
\e
{\Tcal_{n, \leq 1}}_{|\nu_L} \to \Tcal_{{n_1}, \leq 1} \sqcup \cdots  \sqcup \Tcal_{{n_k}, \leq 1}
\e 
and we want that $P$ factors through this map for $L$ large enough. Here, we map $\Tcal_{{n_1}, \leq 1} \sqcup \cdots  \sqcup \Tcal_{{n_k}, \leq 1}$ to $\mathfrak{X}(G)  \sqcup \Hcal\Jcal_t(M)$ by either $P$ or $V$, depending on whether the corresponding subtrees intersect $\mathrm{TopRightEd}( \Tcal)$ or not.

Just as for $\mathfrak{X}^{\Tcal}(G,f)$, the space $\Pcal ert^{\Tcal}(G,f,V;L_0, L_1) $ is nonempty and convex, and is equipped with a projection to the constant part in $X$:
\ea
\Pcal ert^{\Tcal}(G,f,V;L_0, L_1) &\to \Hcal \Jcal_t(M), \\
P &\mapsto \widehat{P}.
\ea

\end{defi}

Alternatively, one can think about perturbations of $\Pcal ert^{\Tcal}(G,f,V;L_0, L_1)$ as maps $h\Tcal \to \mathfrak{X}(G)  \sqcup \left( \Hcal(M)  \times \Jcal(M)  \right) $, with
\e
h\Tcal = \mathrm{LeftEd}(T) \cup   \mathrm{TopRightEd}(T)\times[0,1]
\e
the universal hybrid tree.

If $T = (\tau, l)$ is a metric tree, recall that we  still denot by $T$ the fibre of $\Tcal$ at $T$:
\e
T = \coprod_{e\in \mathrm{Edges}(T)} I_e .
\e
Likewise, let $hT$ stand for the fibre of $h\Tcal$ at $T$:
\e
hT = \coprod_{e\in \mathrm{LeftEd}(T)} I_e  \sqcup \coprod_{e\in \mathrm{TopRightEd}( T)} I_e \times [0,1]  .
\e

\begin{defi}\label{def:hybrid_tree} 

Pick $P\in \Pcal ert^{\Tcal}(G,f,V;L_0, L_1)$, and let $T = (\tau, l)$ be a metric tree.%, written as before $T = \coprod_{e\in \mathrm{Edges}(T)} I_e$.

Restricting $P$ to $hT$ then gives a map $P_T \colon hT \to \mathfrak{X}(G) \sqcup \Hcal\Jcal(M)$.

A \emph{hybrid tree} in $(G;M, L_0, L_1)$ is a pair of maps
\ea
\gamma &\colon \mathrm{LeftEd}(T) \to G , \\
u &\colon    \mathrm{TopRightEd}( T)\times [0,1] \to (M; L_0, L_1),
\ea
(which we will write as $(\gamma,u) \colon hT \to G \sqcup M$) such that:

\begin{itemize}

\item on each left edge $I_e$, the restriction $\gamma_e$ of $\gamma$ is a flowline for $P_T$.

\item on each top-right strip $Z_e = I_e\times [0,1]$, the restriction $u_e$ of $u$ statisfies the Floer equation for $P_T$.

\item at every ``left vertex'' (i.e. those only touching flowlines $\mathrm{LeftEd}(T)$), the usual multiplicative condition holds.

\item at every ``top-right vertex'' $v$ (i.e. the ones touching top-right strips), with the incoming edges at $v$:
\e
 \mathrm{In}(v) = (e_1, \ldots, e_k) ,
\e
and $f$ the outgoing edge, we have:
\e
\forall t\in [0,1], \ \gamma_{e_1}(v) \times \cdots \times u_{e_k}(v, t) = u_{f}(v, t).
\e
\end{itemize}

If, as an ordered set, $\mathrm{Leaves}(\tau)= (l_1, \ldots, l_n)$, let
\e
\lim_{-\infty} {\gamma} = (\lim_{-\infty} {\gamma}_{l_1}, \ldots , \lim_{-\infty} {u}_{l_n}) \in (\mathrm{Crit} f)^{n-1} \times \Ical_{\widehat{P}}(L_0, L_1) ,
\e
and let
\e
\lim_{+\infty} {\gamma} = \lim_{+\infty} {u}_{r_\tau} \in \Ical_{\widehat{P}}(L_0, L_1) ,
\e
with $r_\tau$ the root.
\end{defi}

For $\underline{x}\in (\mathrm{Crit} f)^{n-1} \times \Ical_{\widehat{P}}(L_0, L_1)$ and $y\in \Ical_{\widehat{P}}(L_0, L_1)$, one can define moduli spaces $\Mcal(\underline{x},y; V,P)$ of hybrid trees in $(G;M,L_0, L_1)$ with limits $\underline{x}$ and $y$ at $\pm \infty$.

These are solutions of a Fredholm problem, and we denote $\Mcal(\underline{x},y; V,P)_i$ those of virtual dimension $i$ (i.e. of Fredholm index $i$, except for the case of strips, where there is an additional quotient by $\rr$, in which case the (Maslov) Fredholm index is $i+1$).

For $P$ in a comeagre subset 
\e
\Pcal ert^{\Tcal}_{\mathrm{reg}}(G,f,V;L_0, L_1)\subset \Pcal ert^{\Tcal}(G,f,V;L_0, L_1),
\e 
this space is smooth of dimension given by its Fredholm index. When of dimension 0 it is compact, and when of dimension 1, it can be compactified to $\overline{\Mcal}(\underline{x},y; V,P)$, with boundary
\ea
\partial \overline{\Mcal}(\underline{x}, y ; V,P)&= \bigcup_{z\in\Ical_{\widehat{P}}(L_0, L_1); \ 1\leq i\leq j= n} \Mcal(\underline{x}_1, z ; V,P) \times \Mcal(\underline{x}_2, y ; V,P),\label{eq:boundary_mod_space_module_str2} \\
&\cup  \bigcup_{z\in\mathrm{Crit}f; \ 1\leq i\leq j< n} \Mcal(\underline{x}_1, z ; V) \times \Mcal(\underline{x}_2, y ; V,P), \nonumber
\ea
where (see Figure~\ref{fig:tree_breaking}): 
\begin{itemize}
\item $\underline{x}_1 = (x_i, x_{i+1}, \ldots, x_j)$, $z$ is such that  $\mathrm{vdim}(\underline{x}_1,z)=0$,
\item $\underline{x}_2 = (x_1, \ldots, x_{i-1}, z,  x_{j+1} \ldots, x_n)$ (and therefore $\mathrm{vdim}(\underline{x}_2,y)=0$).
\end{itemize}

In addition to sphere and disk bubbling (ruled out by our assumptions), another kind of bubble can possibly develop when energy concentrates at a seam point, while simultaneously a strip length goes to zero, see Figure~\ref{fig:fig8_bubbling}. From Bottman's removal of singularity theroem \cite{Bottman}, this will either be a figure 8 bubble, or possibly its disc counterpart. In either cases, if one denotes $(b_0,b_1,b_2)$ the three components of the bubble, and $g_0, g_1$ the values of $\gamma_0, \gamma_1$ at the limit, then $(b_0,g_0\cdot b_1,g_0\cdot g_1\cdot b_2)$ will be an actual sphere or disc in $M$. From our assumptions, it will have to be constant.

\begin{figure}[!h]
    \centering
    \def\svgwidth{\textwidth}
   \includegraphics[scale=.15]{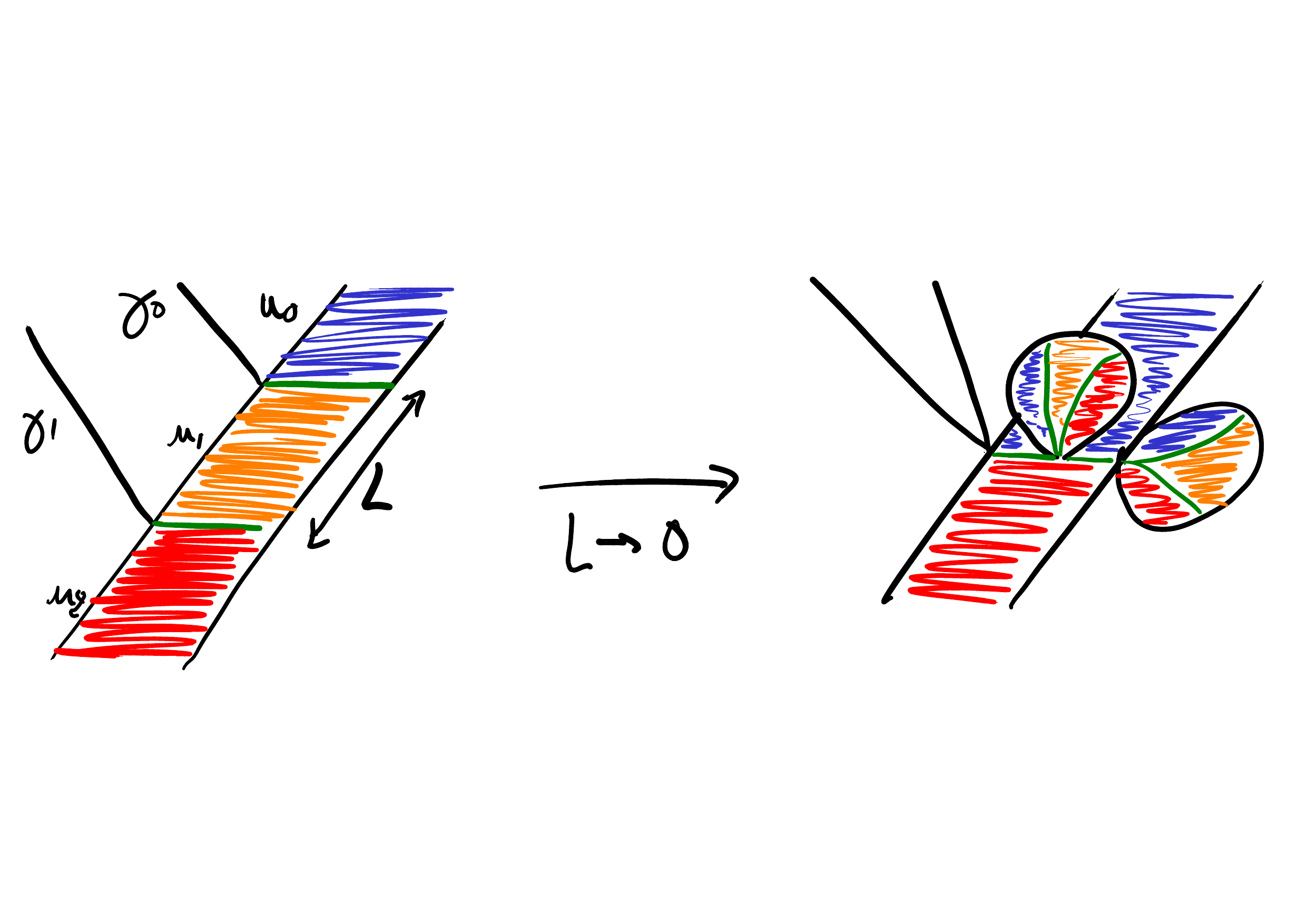}
      \caption{Figure 8 sphere and disc bubling.}
      \label{fig:fig8_bubbling}
\end{figure}

Then, with $A= CM_*(G,f)$ and  $M= CF(L_0, L_1; \widehat{P})$,  the zero-dimensional moduli spaces define maps  
\e
\mu_M^{n|1}\colon A^{\otimes n}\otimes M \to M ,
\e
and equation~\ref{eq:boundary_mod_space_module_str2} shows that these are \Ainf -module structure maps.

\bibliographystyle{alpha}
\bibliography{biblio}

\end{document}